\documentclass[12pt]{article}
\DeclareUnicodeCharacter{FF0C}{,} 
\usepackage{etex}
\usepackage{amsmath,amssymb,amsfonts,amsthm}
\usepackage{a4wide}
\usepackage{color}
\usepackage{verbatim}
\usepackage{epsfig,subfigure,epstopdf}
\usepackage[]{graphics}
\usepackage{graphicx,inputenc}
\usepackage{subfigure}
\usepackage[justification=centering]{caption}
\usepackage{pictex}
\usepackage{wrapfig}
\usepackage{multirow}
\usepackage{etex}
\usepackage{placeins}
\usepackage{graphicx} 
\graphicspath{{Figure/}} 
\usepackage[font=small,labelfont=bf,labelsep=space]{caption} 

\usepackage[T1]{fontenc}
\usepackage{authblk}
\usepackage[textwidth=6.0in,textheight=9in]{geometry}

\usepackage[colorlinks,linktocpage,linkcolor=blue]{hyperref}
\usepackage{fancyhdr}

\newtheoremstyle{exampstyle}
  {\topsep} 
  {\topsep} 
  {\itshape} 
  {} 
  {\bfseries} 
  {.} 
  {.5em} 
  {} 

\theoremstyle{exampstyle}

\numberwithin{equation}{section}
\newtheorem{lemma}{Lemma}[section]

\newtheorem{theorem}{Theorem}[section]

\newtheorem{problem}{Problem}[section]

\newtheorem{remark}{Remark}[section]

\newtheorem{Algorithm}{Algorithm}[section]
\newtheorem{example}{Example}

\allowdisplaybreaks

\usepackage{parskip}
\let\oldref\ref
\renewcommand{\ref}[1]{(\oldref{#1})}  
\renewcommand{\eqref}[1]{(\oldref{#1})}

\newbox\boxaddrone \newbox\boxaddrtwo

\def\N+{n\in\mathbb{N}^{+}}

\def\n{\partial{\overrightarrow{\bf n}}}

\begin{document}

\title{\large\textbf{Determine the point source of the heat equation with sparse boundary measurements}}
\author[1]{Qiling Gu\thanks{guql@sustech.edu.cn}} 
\author[1]{Wenlong Zhang\thanks{Corresponding author, zhangwl@sustech.edu.cn}}
\author[2,3]{Zhidong Zhang\thanks{zhangzhidong@mail.sysu.edu.cn}}
\affil[1]{\normalsize{Department of Mathematics, Southern University of Science and Technology (SUSTech), 1088 Xueyuan Boulevard, University Town of
Shenzhen, Xili, Nanshan, Shenzhen, Guangdong Province, PR China}}
\affil[2]{\normalsize{School of Mathematics (Zhuhai), Sun Yat-sen University, Zhuhai 519082, Guangdong, China}}
\affil[3]{\normalsize{Guangdong Province Key Laboratory of Computational Science, Sun Yat-sen University, Guangzhou 510000, Guangdong, China}}

\maketitle

\begin{abstract}
\sloppy
\noindent  In this work the authors consider the recovery of the point source in the heat equation. The used data is the sparse boundary measurements. The uniqueness theorem of the inverse problem is given. After that, the numerical reconstruction is considered. We propose a numerical method to reconstruct the location of a Dirac point source by reformulating the inverse problem as a least-squares optimization problem, which is efficiently solved using a gradient descent algorithm. Numerical experiments confirm the accuracy of the proposed method and demonstrate its robustness to noise. \\
\fussy
\sloppy
\fussy
\end{abstract}

\section{Introduction.}

\subsection{Mathematical statement.}
In this article, the considered mathematical model is given as:
\begin{equation}
\label{eq-ibvp}
\left\{
\begin{aligned}
(\partial_t-\Delta) u(x,t) &= F(x), && (x,t) \in \Omega\times(0,T),\\
u(x,t) &= 0, &&(x,t)\in\Omega\times\{0\}\cup \partial\Omega\times(0,T) .
\end{aligned}
\right.
\end{equation}
The domain $\Omega$ is set as the unit disc in $\mathbb R^2$; the vanishing boundary and initial conditions are used. The source term $F(x)$ is defined as
\begin{equation*}
 F(x)=\delta_{s_*}=\delta_{(x-x_*,y-y_*)}=\delta_{(r-r_*,\theta-\theta_*)},
\end{equation*}
where the location $(x_*,y_*)$ (or $(r_*,\theta_*)$ in the polar coordinate) of the point source is unknown, and $\delta$ is the Dirac function. In this inverse problem, the used measurements are the boundary flux data, with the formulation 
\begin{equation}\label{data}
\frac{\partial u}{\n} (z,t),\ t\in(0,T),
\ z\in Z_{ob}\subset\partial\Omega.
\end{equation}
The notation $Z_{ob}$ means the observation area on the boundary $\partial\Omega$. Hence, the inverse problem in this work can be stated as follows. 
\begin{problem}
\label{prob-inver}
With the boundary flux data \eqref{data}, 
can we uniquely determine the location of the point source $s_*$?
\end{problem}
We will restrict the size of the observation area $Z_{ob}$ as small as possible. This is the reason we call the data as the sparse boundary data.

\subsection{Background and literature.}
In the practical applications of inverse problems, the control of cost must be one of the core issues. This is the reason that people prefer to use sparse data, which could decrease the expense of equipments, labor, computation and so on.  

Recently, the research on the inverse problems with sparse data has draw more and more attention. In \cite{LiLiuShi:2023}, the authors recover the unknown source in the inverse scattering problems with sparse far-field data; 
\cite{LiYangZhangZhang:2021} considers the inverse obstacle problem with phaseless data; 
\cite{LiuMeng:2023} uses sparse near-field measurements to solve the inverse source problems. For the inverse source problems in the heat equation defined on the unit disc on $\mathbb R^2$, \cite{RundellZhang:2020} uniquely determines the variable separable source with the flux data generated from finite observation points on the boundary; while the authors in \cite{LiZhang:2020} use the same data to recover the semi-discrete source. Furthermore, for the general parabolic equation, \cite{LinZhangZhang:2022, SunZhang:2022} prove that the semi-discrete unknown source can be uniquely determined by the boundary data, where the observation area can be an arbitrary nonempty open subset of the boundary.  

In this work, we set the unknown source as the point source. We could list some applications of this setting. For instance, if equation \eqref{eq-ibvp} is used to describe the diffusion of pollutants \cite{EganMahoney:1972}, it makes sense that we model the factory releasing pollutants as the point source. However, the setting of point source will cause some difficulties in analysis. This is because the point source is not in the space $L^2(\Omega)$, and sequentially we may not follow the techniques in \cite{RundellZhang:2020, LinZhangZhang:2022}.

Several numerical methods have been developed to solve the point source problem. \cite{Ling:2006} investigates the inverse identification of point sources in two-dimensional heat equations. This approach combines the method of fundamental solutions and collocation techniques to reconstruct the locations and intensities of sources from scattered data, utilizing the Levenberg-Marquardt optimization method to improve stability and ensure accuracy. \cite{BenBelgacem:2012} examines the identifiability of pointwise sources in the one-dimensional Fisher’s reaction-diffusion equation, leveraging unique continuation theorems and maximum principles to provide theoretical guarantees for source detection under both steady and moving conditions. \cite{Nakaguchi2012} proposes an algebraic reconstruction algorithm to identify the location and magnitude of a moving point source in a three-dimensional scalar wave equation. The method utilizes observation data and linear algebra techniques to solve the derived algebraic relations. 
\cite{Andrle2011} introduces a numerical method for identifying a moving pointwise source in a one-dimensional advection-dispersion-reaction equation. The methodology involves transforming the original equation into a dispersion-reaction form and solving the inverse problem through either least-squares minimization or the Kohn-Vogelius approach, supported by Crank-Nicolson time discretization and finite element methods for spatial approximation. \cite{Andrle2012} identifies multiple moving pollution sources in an advection-dispersion-reaction equation by reformulating the inverse problem into least-squares and Kohn-Vogelius minimization. This approach utilizes Crank-Nicolson time discretization, finite element methods for spatial discretization, and a quasi-Newton optimization algorithm with the adjoint state method for gradient computation.  \cite{Bruckner2000} reconstructs point wave sources in a one-dimensional wave equation by formulating the inverse problem as a combination of a well-posed part, addressed via Volterra integral equations, and an ill-posed part, regularized using truncated singular value decomposition  and related discretization methods. The method leverages Fourier analysis and eigenfunction expansions to represent the solution, ensuring stability and robustness even in the presence of noisy observations. \cite{Hu2024} develops a singularity enriched neural network  for point source identification in the Poisson equation, leveraging the fundamental solution to capture singularities and neural networks to approximate the regular part. The parameters are optimized through empirical loss minimization via gradient-based methods.  \cite{Kandaswamy2009} presents a non-iterative analytic sensing method that leverages analytic sensors and the reciprocity gap principle to transform the inverse problem of point source localization and intensity estimation into a generalized sampling problem, efficiently solving it with separated linear and nonlinear steps. 
\cite{Faria2020} combines the method of fundamental solutions with a genetic algorithm to reconstruct point sources in the Poisson equation. In this approach, the method of fundamental solutions generates solutions, while the genetic algorithm minimizes the cost function to address the inverse problem. Similarly, \cite{Faria2022} employs the method of fundamental solutions to reconstruct pointwise sources in the modified Helmholtz equation. The algorithm addresses the inverse problem by minimizing a boundary least-squares functional, representing point sources as a linear combination of fundamental solutions. Sensitivity analysis is incorporated to ensure stability, and Simpson’s integration is used for numerical calculations. This meshless method achieves high accuracy and rapid convergence in both two- and three-dimensional cases, even in the presence of noise.
\cite{Baratchart2005} employs rational and meromorphic approximation techniques to locate pointwise sources or small inclusions in two-dimensional domains from boundary measurements. The approach leverages Hardy spaces and best rational approximation to characterize source singularities, utilizing the distribution of poles to efficiently recover source locations and intensities. This method guarantees robust and convergent algorithms with high resolution for inverse problems governed by the Laplace equation. \cite{Chen2022,Wang2023} apply the gradient descent method to solve the least-squares regularization optimization problem for inverse source problems in parabolic equations. \cite{ElBadia2000} investigates an inverse source problem for elliptic equations, utilizing Cauchy problem solvers and algebraic methods to identify the number, locations, and characteristics of monopolar and dipolar sources based on boundary measurements. \cite{Ren2019} reconstructs point sources in the Helmholtz equation within heterogeneous media, leveraging stability estimates and projection-based numerical methods to address the inverse problem with both theoretical rigor and numerical experiments. The study combines adjoint equations and  FEM simulations to estimate the strength and location of point sources in urban environments, validated through wind tunnel experiments and simulations. \cite{Kovalets2011} introduces an algorithm that integrates variational methods with computational fluid dynamics models to estimate the location and intensity of stationary point sources in complex urban environments. The method minimizes a cost function and employs source-receptor functions derived from adjoint equations, achieving high accuracy and robustness in both synthetic and real-world measurement scenarios.
 \cite{Liu2023} addresses the identification of acoustic point sources in a two-layered medium by analyzing sparse far-field patterns across multiple frequencies. A novel direct sampling method is proposed to reconstruct source locations and strengths with theoretical guarantees of stability and uniqueness.
\cite{LeNiliot2004}  employs the boundary element method  and a parameter estimation approach to identify the locations and strengths of point heat sources from boundary measurements, offering insights into the reliability of the results through confidence interval analysis. \cite{Zhang2023} tackles the joint inverse problem of identifying acoustic point sources and obstacles in a two-dimensional domain using Cauchy data. By employing a decomposition strategy and sampling-type imaging methods, the study achieves accurate reconstructions of both the source locations and the obstacle boundary.

\subsection{Main result and outline.}
For Problem \ref{prob-inver}, we prove the uniqueness theorem and give a positive answer. Meanwhile, the size of measured area $Z_{ob}$ is limited to two appropriately
chosen points, i.e. $Z_{ob}=\{z_1,z_2\}\subset\partial\Omega$.
This reflects the sparsity in the title. More precisely, we state the uniqueness theorem as follows, which is the main result of this work. 
\begin{theorem}
\label{thm-unique}
Set $z_\ell=(\cos\theta_\ell,\sin\theta_\ell)\in \partial\Omega,\ \ell=1,2$
 to be the boundary observation points and suppose the following condition is fulfilled,
\begin{equation}
\label{condi-z}
\theta_1 - \theta_2 \not\in \pi\mathbb Q,\ \mathbb Q
\ \text{is the set of rational numbers.}
\end{equation}
Given two unknown point source $(r_*,\theta_*)$ and $(\tilde r_*,\tilde \theta_*)$, we denote corresponding solutions of equation \eqref{eq-ibvp} by $u$ and $\tilde u$, respectively. If 
\begin{equation*}
\label{condi-obser}
\frac{\partial u}{\n} ( z_\ell,t)= \frac{\partial \tilde u}{\n} (z_\ell,t),\ t\in(0,T),\ \ell=1,2,
\end{equation*}
then
$$r_*=\tilde r_*,\ \theta_*=\tilde\theta_*.$$
\end{theorem}

The rest of this article is structured as follows. In section \ref{section_pre}, we collect several preliminary results, which contain the eigensystem of the Laplacian $\Delta$ on the unit disc in $\mathbb R^2$, and the auxiliary set $\{\xi_l\}$ of harmonic functions. In section \ref{section_unique}, we give the proof of Theorem \ref{thm-unique}. After that, we consider the numerical reconstructions of the inverse problems in sections \ref{section_numerical} and \ref{section_example}. In section \ref{section_numerical}, numerical algorithms are developed to reconstruct the Dirac point source. The forward problem is solved by reformulating the heat equation in polar coordinates, employing finite element methods for spatial discretization and the backward Euler scheme for time integration, while the inverse problem is addressed using a gradient descent optimization approach. Finally, in section \ref{section_example}, some numerical experiments are presented to validate the effectiveness and robustness of the proposed method.

\section{Preliminaries.}\label{section_pre}
In this section, we collect several preliminary knowledge for the future proofs. 
\subsection{\texorpdfstring{Dirichlet eigensystem of $-\Delta$.}{Dirichlet eigensystem of -Delta}} 
We denote the eigensystem of the operator
$-\Delta$ on $\Omega$ with Dirichlet boundary condition 
by $\{(\lambda_n,\varphi_n)\}_{n=1}^\infty$
(multiplicity counted). Due to the self-adjointness of $-\Delta$, we have that 
$$
0<\lambda_1\le \cdots \le \lambda_n\le \cdots
\to\infty,\quad \mbox{as $n\to \infty$,}
$$
and $\{\varphi_n\}_{n=1}^\infty$ form an orthonormal basis of $L^2(\Omega)$. 
Also, \cite{GrebenkovNguyen:2013} gives the representations of the eigenfunctions $\{\varphi_n\}_{n=1}^\infty$ as 
\begin{equation}\label{eigenfunction}
\varphi_n(r,\theta)=\omega_nJ_{|m(n)|}(\lambda_n^{1/2}r)e^{im(n)\theta },
\ n\in\mathbb N^+.
\end{equation}
Here $(r,\theta)$ are the polar coordinates on $\Omega$, $\{\omega_n\}$ are the normalized coefficients to make sure
$\|\varphi_n\|_{L^2(\Omega)}=1$, and
$J_{|m(n)|}(\cdot)$ is the Bessel function of order $|m(n)|$
with $\lambda_n^{1/2}$ as its zero point. The Bessel orders $m$
depend on the choice of $n$ and we use the notation $m(n)$ to show
the dependence (sometimes we may use $J_m$ for short).
The next remark can be found in \cite{LiZhang:2020}, which concerns the structure of eigenfunctions \eqref{eigenfunction}.
\begin{remark}
From Bourget's hypothesis, proved in \cite{Siegel:2014},
there exist no common positive zeros between two Bessel functions
with different nonnegative integer orders. Also recall that
$J_{-m}(r)=(-1)^mJ_m(r)$, given an
eigenvalue $\lambda_{n_0}$, $\lambda_{n_0}^{1/2}$ can only be the zero
of $J_{\pm m(n_0)}(\cdot)$. Hence the multiplicity for $\lambda_{n_0}$
is two if $m(n_0)$ is nonzero, otherwise, it will be one.

In the case of $m(n_0)\ne 0$, by setting
$\lambda_{n_0}=\lambda_{n_0+1}$, the corresponding eigenpairs are given as
$$(\lambda_{n_0}, \omega_{n_0}J_{|m(n_0)|}(\lambda_{n_0}^{1/2}r)e^{i|m(n_0)|\theta }),
\quad (\lambda_{n_0+1}, \omega_{n_0+1}J_{|m(n_0)|}(\lambda_{n_0+1}^{1/2}r)e^{-i|m(n_0)|\theta }).$$
Now setting $m(n_0)=|m(n_0)|=-m(n_0+1)$, the representation \eqref{eigenfunction} is consistency and $m$ is uniquely determined by the value of $n$. See \cite{GrebenkovNguyen:2013} for details about the structure of $\{\varphi_n\}_{n=1}^\infty$.
\end{remark}

The next lemma concerns the values of the normalized parameters $\{\omega_n\}$. 
\begin{lemma}\label{omega}
The normalization coefficients $\omega_n$ in eigenfunctions \eqref{eigenfunction} satisfy 
$$\omega_n = \pi^{-1/2} \left[J_{|m(n)|+1}(\lambda_n^{1/2}) \right]^{-1},\  n\in\mathbb N^+.$$
\end{lemma}
\begin{proof}
The proof follows from \cite[Lemma 2.6]{LiZhang:2020}. 
\end{proof}

\subsection{\texorpdfstring{The set of harmonic functions $\{\xi_l\}_{l=-\infty}^\infty$.}{\protect The set of harmonic functions xi l for l from -infinity to infinity}}
We start to build the set $\{\xi_l\}_{l=-\infty}^\infty$, which would be used to represent the boundary measurements. 

For $l \in\mathbb Z$, we define the harmonic functions
$$
\xi_l(x) = \xi_l(r,\theta) = 2^{-1/2}\pi^{-1/2} r^{|l|} e^{i l\theta}.
$$
Setting $r=1$, that is, $|x|=1$, it is not difficult to check that the
functions $\{\xi_l(1,\theta)\}_{l=-\infty}^\infty$ form an
orthonormal basis of $L^2(\partial\Omega)$. For
$z=(\cos\theta_z,\sin\theta_z)\in\partial\Omega$,
we define $\eta_z^N\in C^\infty(\overline\Omega)$ as
\begin{equation}
\label{defi-psi}
\eta_z^N(x) := \sum_{l=-N}^N \xi_l(z) \xi_{-l}(r,\theta),\quad (r,\theta)\in[0,1]\times[0,2\pi),
\end{equation}
and state the next lemma. 

\begin{lemma}
\label{lem-appro-delta}
Let $u$ satisfy equation \eqref{eq-ibvp}. For $ z\in\partial\Omega$, we have
$$
\lim_{N\to\infty} \int_{\partial\Omega} \eta_z^N (x)
\frac{\partial u}{\n} (x,t)\ dx =
\frac{\partial u}{\n} (z,t),\quad a.e.\ t\in(0,T).
$$
\end{lemma}
\begin{proof}
From the definition \eqref{defi-psi} of the function $\eta_z^N$, it follows that
\begin{align*}
\sum_{l=-N}^N \int_{\partial\Omega} \xi_l(z) \xi_{-l}(x)
\frac{\partial u}{\n} (x,t)\ dx
=\sum_{l=-N}^N \langle \frac{\partial u}{\n} (\cdot,t),
\xi_l(\cdot)  \rangle_{L^2(\partial\Omega)}\ \xi_l(z),\quad t\in(0,T).
\end{align*}
Here $\langle\cdot, \cdot \rangle_{L^2(\partial\Omega)}$ is the inner product of $L^2(\partial\Omega)$. 
The smoothness property of $ \frac{\partial u}{\n} (\cdot,t)$, which is proved in Appendix, leads to 
the pointwise convergence of the Fourier series of
$\frac{\partial u}{ \n}(\cdot,t)$. Hence, we have 
$$\lim_{N\to\infty}\sum_{l=-N}^N \langle \frac{\partial u}{\n} (\cdot,t),
\xi_l(\cdot)  \rangle_{L^2(\partial\Omega)}\ \xi_l(z)=\frac{\partial u}{\n} (z,t),
\quad t\in(0,T),$$
which completes the proof.
\end{proof}

Obviously $\eta_z^N\in L^2(\Omega)$, then we can express $\eta_z^N$ in $L^2(\Omega)$ sense as
\begin{equation*}
\label{eq-delN}
\eta_z^N(x) = \sum_{n=1}^\infty \langle \eta_z^N,
\varphi_n\rangle_{L^2(\Omega)} \varphi_n(x).
\end{equation*}
By \eqref{eigenfunction} and Lemma \ref{omega}, the Fourier coefficients $\langle \eta_z^N, \varphi_n \rangle_{L^2(\Omega)}$
are calculated as
\begin{equation}\label{a_n}
\langle \eta_z^N, \varphi_n\rangle_{L^2(\Omega)}=
\begin{cases}
\pi^{-1/2}\lambda_n^{-1/2} e^{-im(n)\theta_z}, & \mbox{if $N\ge|m(n)|$,}
\\
0, & \mbox{otherwise.}
\end{cases}
\end{equation}

We assume $u_z^N$ is the solution of the following initial-boundary value problem
\begin{equation}
\label{eq-uM}
\left\{
\begin{aligned}
(\partial_t-\Delta) u_z^N(x,t) &=0, && (x,t)\in \Omega\times(0,T),\\
u_z^N(x,t) &= 0,&& (x,t)\in\partial\Omega\times (0,T),\\
u_z^N(x,t) &=-\eta_z^N(x), && (x,t)\in\Omega\times\{0\}.
\end{aligned}
\right.
\end{equation}
In view of the fact that $\eta_z^N$ is the linear combination of harmonic functions on $\Omega$, we see that $\Delta \eta_z^N = 0$, then $w_z^N(x,t):=u_z^N(x,t) + \eta_z^N(x)$
satisfies the following initial-boundary value problem
\begin{equation}
\label{eq-adjoint}
\left\{
\begin{alignedat}{2}
(\partial_t-\Delta) w_z^N(x,t) &=0, &\quad& (x,t)\in \Omega\times(0,T),\\
w_z^N(x,t) &= \eta_z^N(x),&\quad& (x,t)\in\partial\Omega\times (0,T),\\
w_z^N(x,t) &= 0, &\quad& (x,t)\in\Omega\times\{0\}.
\end{alignedat}
\right.
\end{equation}

Using \eqref{a_n}, \eqref{eq-uM} and the property of heat equation, the representation of $w_z^N$ in \eqref{eq-adjoint} can be given as 
\begin{equation}\label{w_z^N}
w_z^N(x,t) = \sum_{|m(n)|\le N} \pi^{-1/2}\lambda_n^{-1/2} e^{-im(n)\theta_z}\big(1
-e^{-\lambda_nt} \big)\varphi_n(x),\ (x,t)\in \Omega\times(0,T).
\end{equation}

\section{The proof of Theorem \ref{thm-unique}. }\label{section_unique}
Now we will establish the proof of Theorem \ref{thm-unique}.

\subsection{Representation of the boundary flux data.}

In this subsection we will build a connection between the flux measurements
and the unknowns.
\begin{lemma}\label{lem_measurement_1}
Assume $z\in\partial\Omega$, and let $u$ and $w_z^N$ be the solutions
of \eqref{eq-ibvp} and \eqref{eq-adjoint} respectively, then
$$
- \frac{\partial u}{\n}(z,t)=\lim_{N\to \infty}\int_\Omega F(x) w_z^N(x,t)\ dx, 
\quad  t\in(0,T).
$$
\end{lemma}
\begin{proof}
Equation \eqref{eq-adjoint} and Green's identities yield that
for $v\in H_0^1(\Omega)$,
$$
\int_\Omega \partial_t w_z^N(x,t) v(x) + \nabla w_z^N(x,t)\cdot\nabla v(x)\ dx = 0,\quad  t\in(0,T).
$$
On the other hand, taking convolution of $w_z^N(x,t)$ and equation \eqref{eq-ibvp}, we see that
\begin{align*}
I_N(t) :=& \int_0^t\int_\Omega F(x) w_z^N(x,t-\tau)\ dx\ d\tau\\
=& \int_0^t\int_\Omega \big[ \partial_t u(x,\tau)
- \Delta u(x,\tau) \big] w_z^N(x,t-\tau)\ dx\ d\tau.
\end{align*}
With Green's identities, we have
\begin{align*}
I_N(t)  = &\int_0^t\int_\Omega \partial_t u(x,\tau) w_z^N(x,t-\tau)\ dx\ d\tau
+ \int_0^t\int_\Omega \nabla u(x,\tau) \cdot \nabla w_z^N(x,t-\tau)\ dx\ d\tau\\
&- \int_0^t\int_{\partial\Omega} \frac{\partial u}{\n}(x,\tau) w_z^N(x,t-\tau)\ dx\ d\tau.
\end{align*}
With \eqref{w_z^N} and $w_z^N(x,0)=0$, we arrive at the equality
$$
 \int_0^t \int_\Omega \partial_t u(x,\tau) w_z^N(x,t-\tau)
 \ dx\ d\tau
=  \int_0^t \int_\Omega u(x,\tau) \partial_t w_z^N(x,t-\tau)\ dx \ d\tau.
$$
Finally, we get
\begin{align*}
I_N  = &
\int_0^t\int_\Omega \Big[\partial_t w_z^N(x,t-\tau) u(x,\tau)
+ \nabla w_z^N(x,t-\tau)\cdot \nabla u(x,\tau)\Big]\ dx\ d\tau
\\
&- \int_0^t \int_{\partial\Omega} \frac{\partial u}{\n}(x,\tau) \eta_z^N(x)\ dx\ d\tau\\
=&- \int_0^t\int_{\partial\Omega} \frac{\partial u}{\n}(x,\tau) \eta_z^N(x)\ dx\ d\tau.
\end{align*}
For the above result, differentiating on time $t$ gives that 
$$\int_\Omega F(x) w_z^N(x,t)\ dx=- \int_{\partial\Omega} \frac{\partial u}{\n}(x,t) \eta_z^N(x)\ dx,\  t\in(0,T).$$
With Lemma \ref{lem-appro-delta} we have 
$$
 \lim_{N\to \infty}\int_\Omega F(x) w_z^N(x,t)\ dx
=-\frac{\partial u}{\n}(z,t),\  t\in(0,T),
$$
which completes the proof.
\end{proof}

With the above lemma, we can show the result below straightforwardly.
\begin{lemma}
\label{lem-ipfp}
For $z=(\cos \theta_z,\sin\theta_z)\in\partial\Omega$, we have for $ t\in(0,T)$,
\begin{align*}
 \frac{\partial u}{\n}(z,t) &=\sum_{n=1}^\infty a_{n,s_*}e^{-im(n)\theta_z}\big(1-e^{-\lambda_nt}\big)\\
 &=\sum_{n=1}^\infty a_{n,r_*}e^{-im(n)(\theta_z-\theta_*)}\big(1-e^{-\lambda_nt}\big),
\end{align*}
where the coefficients $\{a_{n,r_*},a_{n,s_*}\}$ are given as 
\begin{align*}
a_{n,r_*}&=-\pi^{-1/2}\lambda_n^{-1/2} \omega_nJ_{|m(n)|}(\lambda_n^{1/2}r_*),\\
 a_{n,s_*}&=a_{n,r_*}e^{im(n)\theta_*}.
\end{align*}
\end{lemma}
\begin{proof}
We have
\begin{equation*}
\begin{aligned}
\int_\Omega F(x) w_z^N(x,t)\ dx
&=w_z^N(s_*,t).
\end{aligned}
\end{equation*}
So, letting $N\to \infty$, \eqref{w_z^N} leads to 
\begin{align*}
 \frac{\partial u}{\n}(z,t) &= -\sum_{n=1}^\infty \pi^{-1/2}\lambda_n^{-1/2} \omega_nJ_{|m(n)|}(\lambda_n^{1/2}r_*)e^{im(n)(\theta_* -\theta_z)}\big(1-e^{-\lambda_nt} \big),
\end{align*}
which completes the proof.
\end{proof}

\subsection{Auxiliary results.}
In this subsection, we list several auxiliary results, which will be used in the proof of the uniqueness theorem. 
\begin{lemma}
\label{lemma_uniqueness_*}
Define $z_\ell:=(\cos{\theta_\ell},\sin{\theta_\ell})\in \partial\Omega, \ \ell=1,2$ satisfying condition \eqref{condi-z}.
Then $ a_{n,r_*}=a_{n,s_*}=0,\ n\in\mathbb{N}^+,$ provided that
$$
\sum_{\lambda_n=\lambda_j} a_{n,r_*}e^{-im(n)(\theta_\ell-\theta_*)}=0, \ j\in\mathbb{N}^+,\ \ell=1,2.
$$
\end{lemma}
\begin{proof}
 Given $j\in \mathbb{N}^+$, if $m(n(j))\ne 0$,
 letting $n(j),\ n(j)+1$ be the integers such that
 $\lambda_n=\lambda_j$, then
 \begin{equation*}
  \sum_{\lambda_n=\lambda_j} a_{n,r_*}e^{-im(n)(\theta_\ell-\theta_*)}
  =e^{i|m|(\theta_\ell-\theta_*)}a_{n(j),r_*}+
  e^{-i|m|(\theta_\ell-\theta_*)}a_{n(j)+1,r_*}=0,\quad \ell=1,2,
 \end{equation*}
which gives
\begin{equation*}
 \begin{bmatrix}
  e^{i|m|(\theta_1-\theta_*)}& e^{-i|m|(\theta_1-\theta_*)}\\
  e^{i|m|(\theta_2-\theta_*)}& e^{-i|m|(\theta_2-\theta_*)}
 \end{bmatrix}
  \begin{bmatrix}
  a_{n(j),r_*}\\a_{n(j)+1,r_*}
 \end{bmatrix}
 =\begin{bmatrix}
   0\\0
  \end{bmatrix}.
\end{equation*}
The determinant of the matrix is
\begin{equation*}
 e^{i|m|(\theta_1-\theta_2)}-e^{-i|m|(\theta_1-\theta_2)}
 =2i\sin{(|m|(\theta_1-\theta_2))}\ne 0,
\end{equation*}
by condition \eqref{condi-z} and $m\ne 0$.
Hence, we have $  a_{n(j),r_*}=a_{n(j)+1,r_*}=0$.

In the case of $m(n(j))=0$, it holds that
 \begin{equation*}
  \sum_{\lambda_n=\lambda_j} a_{n,s_*}e^{-im(n)\theta_\ell}
  =a_{n,r_*}=0,\quad \ell=1,2,
 \end{equation*}
sequentially $a_{n,r_*}=0$. Also, recalling that $ a_{n,s_*}=a_{n,r_*}e^{im(n)\theta_*}$, we have $a_{n,s_*}=0$. Since $j$ is chosen arbitrarily, the proof is complete.
\end{proof}

The proof of the next lemma follows from \cite[Lemma 3.5]{RundellZhang:2020}.
\begin{lemma}\label{lemma_uniqueness_Dirichlet}
 Let $\{\tau_n:\N+\}$ be an absolutely convergent complex sequence and
 $\{\gamma_n:\N+\}$ be a real sequence satisfying
 $0\le \gamma_1<\gamma_2<\cdots,\ \gamma_n\to \infty.$
For the complex series $\sum_{n=1}^\infty \tau_n e^{-\gamma_n t}$ which
is defined on $\mathbb{C}^+$, if the set of its zeros on $\mathbb{C}^+$
has an accumulation point, then $\tau_n=0,\ \N+$.
\end{lemma}

\subsection{Proof of Theorem \ref{thm-unique}.}
Now we are ready to prove the main theorem.
\begin{proof}[Proof of Theorem \ref{thm-unique}]
To show the uniqueness theorem, firstly, for $\ell=1,2,\ t\in(0,T)$ we have
\begin{equation*}
 \sum_{n=1}^\infty a_{n,s_*}e^{-im(n)\theta_\ell}\big(1-e^{-\lambda_nt}\big)
 = \sum_{n=1}^\infty a_{n,\tilde s_*}e^{-im(n)\theta_\ell}\big(1-e^{-\lambda_nt}\big).
\end{equation*}
Then by Lemma \ref{lemma_uniqueness_Dirichlet}, we have that
\begin{equation*}
\sum_{\lambda_n=\lambda_j} a_{n,s_*}e^{-im(n)\theta_\ell}=\sum_{\lambda_n=\lambda_j} a_{n,\tilde s_*}e^{-im(n)\theta_\ell},\ \ell=1,2,\ j\in\mathbb N^+.
\end{equation*}
Next, with Lemma \ref{lemma_uniqueness_*}, it yields that
\begin{equation*}
  a_{n,s_*}= a_{n,\tilde s_*},\ \N+,
\end{equation*}
which leads to for $\N+,$
\begin{equation}\label{result}
  J_{|m(n)|}(\lambda_n^{1/2}r_*)e^{im(n)\theta_*}=J_{|m(n)|}(\lambda_n^{1/2}\tilde r_*)e^{im(n)\tilde\theta_*}.
\end{equation}
The above result gives that 
\begin{equation*}
|J_{|m(n)|}(\lambda_n^{1/2}r_*)|=|J_{|m(n)|}(\lambda_n^{1/2}\tilde r_*)|\ \text{for}\ \N+.
\end{equation*}

Recall that the first eigenvalue $\lambda_1$ is the square of the first positive root of the Bessel function $J_0(\cdot)$. Also, from the straightforward computation, we have
\begin{equation*}
 [J_0(x)]'=-J_1(x),\quad J_1(x)>0\ \text{on}\ (0,\lambda_1^{1/2}).
\end{equation*}
The above results give that $J_0$ is strictly positive and decreasing on $(0,\lambda_1^{1/2})$, which together with  $|J_0(\lambda_1^{1/2}r_*)|=|J_0(\lambda_1^{1/2}\tilde r_*)|$ gives
$$r_*=\tilde r_*,\ J_0(\lambda_1^{1/2}r_*)=J_0(\lambda_1^{1/2}\tilde r_*).$$
Inserting the above result into \eqref{result} yield that $\theta_*=\tilde\theta_*$. The proof is complete.
\end{proof}

\section{Numerical algorithms.}\label{section_numerical}
To better accommodate natural boundary conditions, we reformulate  model \eqref{eq-ibvp} in a circular domain and provide its equivalent form in polar coordinates.

The equivalent form of model \eqref{eq-ibvp} in polar coordinates is given by the parabolic equation:
\begin{equation}\label{41}
\partial_t u-\left(\frac{1}{r} \frac{\partial}{\partial r}\left(r \frac{\partial u}{\partial r}\right)+\frac{1}{r^2} \frac{\partial^2 u}{\partial \theta^2}\right)=F(r, \theta),
\end{equation}
where the source term $F(r, \theta)$ is defined as:
$$
F(r, \theta)=\frac{\delta_{\left(r-r_*, \theta-\theta_*\right)}}{r}.
$$
The boundary and initial conditions are:
$$
u(r, \theta, t)=0, \quad(r, \theta, t) \in \Omega \times\{0\} \cup \partial \Omega \times(0, T) .
$$
This formulation aligns with the geometry of circular domains, improving the accuracy of boundary condition enforcement.

\subsection{Forward solver. }
We divide the region $\Omega$ into  non overlapping quadrilateral domains $\Omega_h$ and use quadrilateral finite elements in the spatial direction. For the $i$-th subdomain $\Omega_{h_i}:\left[r_i, r_{i+1}\right] \times$ $\left[\theta_i, \theta_{i+1}\right]$, the coordinate transformations $\xi=2 \cdot\left(r-r_i\right) /\left(r_{i+1}-r_i\right)-1$ and $\eta=2 \cdot\left(\theta-\theta_i\right) /\left(\theta_{i+1}-\theta_i\right)-1$ convert the subdomain to the computational domain $[-1,1] \times[-1,1]$.  In this reference element, the basis functions $N_i(\xi, \eta)$ are linear and are associated with the nodes $(\xi, \eta)$ located at $(-1,-1),(1,-1),(1,1)$, and $(-1,1)$. The basis functions for the quadrilateral element are:
$$
\begin{aligned}
& N_1(\xi, \eta)=\frac{1}{4}(1-\xi)(1-\eta), \\
& N_2(\xi, \eta)=\frac{1}{4}(1+\xi)(1-\eta), \\
& N_3(\xi, \eta)=\frac{1}{4}(1+\xi)(1+\eta), \\
& N_4(\xi, \eta)=\frac{1}{4}(1-\xi)(1+\eta).
\end{aligned}
$$
For a given quadrilateral element $\left[r_i, r_{i+1}\right] \times\left[\theta_i, \theta_{i+1}\right]$, using isoparametric mapping, we can define the relationship between the physical coordinates $(r, \theta)$ and the reference coordinates $(\xi, \eta)$ as follows:
$$
r(\xi, \eta)=\sum_{i=1}^4 N_i(\xi, \eta) r_i, \quad \theta(\xi, \eta)=\sum_{i=1}^4 N_i(\xi, \eta) \theta_i
$$
where $r_i$ and $\theta_i$ are the radial and angular coordinates of the four nodes of the physical element, respectively.
Apply an appropriate time-stepping backward Euler scheme  to discretize the time derivative $\partial_t u$, e.g.
$\frac{\partial u}{\partial t} \approx \frac{u^{n+1}-u^n}{\Delta t}
$.
The weak form of the parabolic equation is obtained by multiplying the equation \eqref{41} by a test function $v \in V=H_0^1(\Omega)$
 and integrating over the domain
\begin{equation}
\int_{\Omega}\left(\frac{u^{n+1}-u^n}{\Delta t} v+\left(\frac{1}{r} \frac{\partial u^{n+1}}{\partial r} \frac{\partial v}{\partial r}+\frac{1}{r^2} \frac{\partial^2 u^{n+1}}{\partial \theta^2} v\right)\right) r \ d r\ d \theta=\int_{\Omega}\frac{\delta_{\left(r-r_1, \theta-\theta_1\right)}}{r} v r\ d r\ d \theta.
\end{equation}
Given the properties of the Dirac delta function, the integral on the right-hand side simplifies to:
\begin{equation}
\int_{\Omega} \frac{\delta_{\left(r-r_1, \theta-\theta_1\right)}}{r} v r\ d r\ d \theta=v\left(r_1, \theta_1\right).
\end{equation}
Therefore, the weak form becomes:
\begin{equation}\label{44}
\int_{\Omega}\left(\frac{u^{n+1}-u^n}{\Delta t} v+\left(\frac{1}{r} \frac{\partial u^{n+1}}{\partial r} \frac{\partial v}{\partial r}+\frac{1}{r^2} \frac{\partial^2 u^{n+1}}{\partial \theta^2} v\right)\right) r\ d r\ d \theta=v(r_1, \theta_1).
\end{equation}

The finite element space $V_h$ are defined by $V_h= span\{N1,N2,N3,N4\}.$
Then, the fully discrete finite element approximation scheme of \eqref{44} is to find $u_h\in V_h$,
\begin{equation}\label{45}
\int_{\Omega}\left(\frac{u_h^{n+1}-u_h^n}{\Delta t} v_h+\left(\frac{1}{r} \frac{\partial u_h^{n+1}}{\partial r} \frac{\partial v_h}{\partial r}+\frac{1}{r^2} \frac{\partial^2 u_h^{n+1}}{\partial \theta^2} v_h\right)\right) r \ d r\ d \theta=v_h(r_1, \theta_1),\ \forall v_h\in V_h .
\end{equation}

\begin{remark}
In solving the parabolic point source problem using the finite element method in a circular domain, aside from the Dirichlet boundary conditions specified in the problem, it is important to address an additional hidden periodic boundary condition. This condition arises at $\theta=0$ and $\theta=2\pi$,  where the values must match to maintain periodicity.
\end{remark}

\subsection{Inverse solver.}
We now develop methods to reconstruct a  Dirac point source. The methods we propose
are applied in a subdomain bounded in space by our observations.
The algorithm for the least-squares problem is summarized as follows. We reformulate the inverse
problem into equivalent minimization problems: a least-squares problem.

The least-squares approach consists in finding the minimum of the objective function $J$
defined as
\begin{equation}\label{46}
\begin{split}
&J(r_1,\theta_1)= \sum_{j=1}^{2}\int_0^T\left(\frac{\partial u_h}{\partial \overrightarrow{\mathbf{n}}}(z_j, t)-f_j^{\delta}\right)^2\ dt, ~t \in(0, T),~ z_1, z_2 \in Z_{o b} \subset \partial \Omega.
\end{split}
\end{equation}

\begin{Algorithm}\label{A41}

 Finite element method and gradient descent method for inversion of a Dirac point source problem.
\hrule
\begin{description}
  \item[1] Solving the forward problem \eqref{45} to produce $f_j^{\delta}=\frac{\partial u_h(z_j,t_i)}{\partial \overrightarrow{\mathbf{n}}}*(1+rand(2,n)*\delta)$, where $j=1,2$, $i=1,2,...,n$, $\delta$ is noise level.

       \item[2] To solve the least squares problem \eqref{46}, we use the gradient descent method.
       \begin{enumerate}
  \item Giving initial values $r_0$ and $\theta_0$.

 \item While $\frac{\|r_k-r_{k-1}\|}{\|r_k\|}+\frac{\|\theta_k-\theta_{k-1}\|}{\|\theta_k\|}<\epsilon$, where $\epsilon$ is the threshold

  Computering $r_k=r_{k-1}-\alpha\frac{\partial J(r_{k-1},\theta_{k-1})}{\partial r_{k-1}}$ and $\theta_k=\theta_{k-1}-\alpha\frac{\partial J(r_{k-1},\theta_{k-1})}{\partial \theta_{k-1}}$.
      \item $r_1=r_k,~~~~\theta_1=\theta_k.$
 \end{enumerate}
\end{description}
\hrule
\end{Algorithm}

\begin{remark}\label{R41}
To determine the step sizes $\alpha$ in step 2.2, we use a linear search. The derivatives of $J$ with respect to $r_{k-1}$ and $\theta_{k-1}$
are given by:
$$\frac{\partial J(r_{k-1},\theta_{k-1})}{\partial r_{k-1}}=2\left(\frac{T}{n}\sum_{j=1}^{2}\sum_{i=1}^{n}F_{ij}\left( \frac{\partial }{\partial \overrightarrow{\mathbf{n}}}\frac{\partial u_h(z_j, t_i,r_{k-1},\theta_{k-1})}{\partial r_{k-1}}\right)\right ),$$
and
$$\frac{\partial J(r_{k-1},\theta_{k-1})}{\partial \theta_{k-1}}=2\left(\frac{T}{n}\sum_{j=1}^{2}\sum_{i=1}^{n}F_{ij}\left( \frac{\partial }{\partial \overrightarrow{\mathbf{n}}}\frac{\partial u_h(z_j, t_i,r_{k-1},\theta_{k-1})}{\partial \theta_{k-1}}\right)\right ), $$where $F_{ij}=\frac{\partial u_h(z_j, t_i,r_{k-1},\theta_{k-1})}{\partial \overrightarrow{\mathbf{n}}}-f_j^{\delta}$.

For circular domains, the boundary derivatives can be computed as follows  $$\frac{\partial u_h(r_{k-1},\theta_{k-1})}{\partial \overrightarrow{\mathbf{n}}}=\frac{\partial u_h(r_{k-1},\theta_{k-1})}{\partial r},$$  where the radial derivative can be approximated using finite element interpolation: $$\frac{\partial v}{\partial r}=\sum_{i=1}^4 \frac{\partial N_i(\xi, \eta)}{\partial r}  v_i,$$ with $v_i$ being the degree of freedom at each node of the element.

To calculate
$\frac{\partial u_h(r_{k-1},\theta_{k-1})}{\partial r_{k-1}}$ and $\frac{\partial u_h(r_{k-1},\theta_{k-1})}{\partial \theta_{k-1}}$, we take the computation of
$\frac{\partial u_h(r_{k-1},\theta_{k-1})}{\partial r_{k-1}}$ as an example.
This involves solving the following equation:
\[
\int_\Omega \left( \frac{z_h^{n+1} - z_h^n}{\Delta t} v_h + \left( \frac{1}{r} \frac{\partial z_h^{n+1}}{\partial r} \frac{\partial v_h}{\partial r} + \frac{1}{r^2} \frac{\partial^2 z_h^{n+1}}{\partial \theta^2} v_h \right) \right) r\ dr\ d\theta = \frac{\partial u_h(r_{k-1}, \theta_{k-1})}{\partial r}, 
\]
where  $\forall v_h \in V_h$, \( z_h = \frac{\partial u_h}{\partial r_{k-1}} \).
\end{remark}

\section{Numerical experiments. }\label{section_example}
In this section, we present a numerical example  to demonstrate the theoretical analysis.
\begin{example}\label{exam1}
The parameter
$m$ represents the number of subdivisions for the radius $r \in[0,1]$, $n$ denotes the number of subdivisions for the angle $\theta \in[0,2 \pi]$,
and $d$ is the number of subdivisions for the time interval $[0,T]$ with $T=1$.  To more clearly visualize the position of the point$'$s coordinates, we convert the polar coordinates into Cartesian coordinates as follows: $x_1=r_1cos(\theta_1)$, $y_1=r_1sin(\theta_1)$.
\end{example}

First, we set $m=n=70$ and  $d=500$ to solve the  forward problem \eqref{45}, thereby generating the data $f_j^{\delta}$. Figure \ref{F1} provides the  side view and vertical view of  the numerical solution $u_h$ for this forward problem.
\begin{figure}
\centering
\subfigure[Side View ]{
\includegraphics[width=6cm]{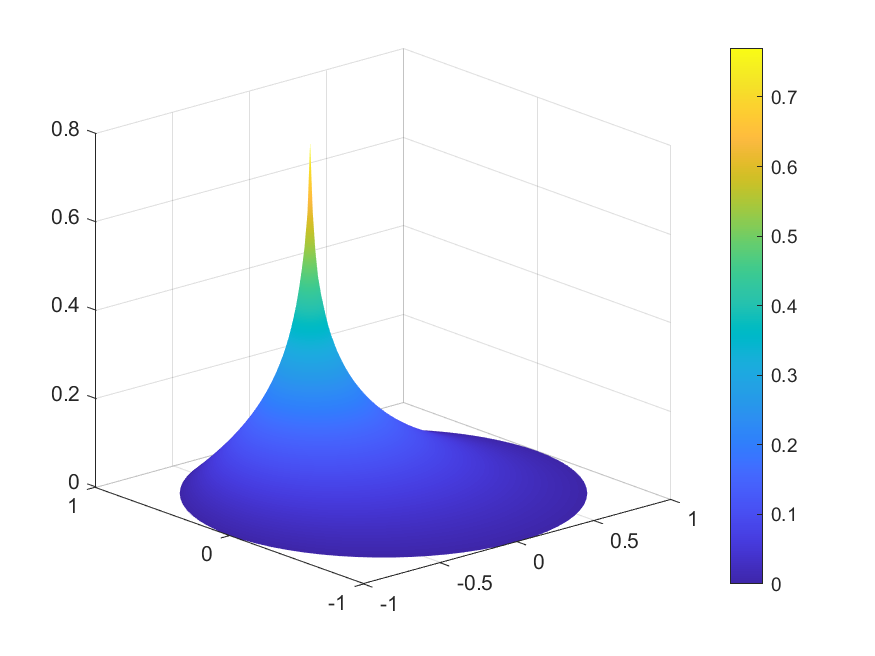}
}
\subfigure[Vertical view]{
\includegraphics[width=6cm]{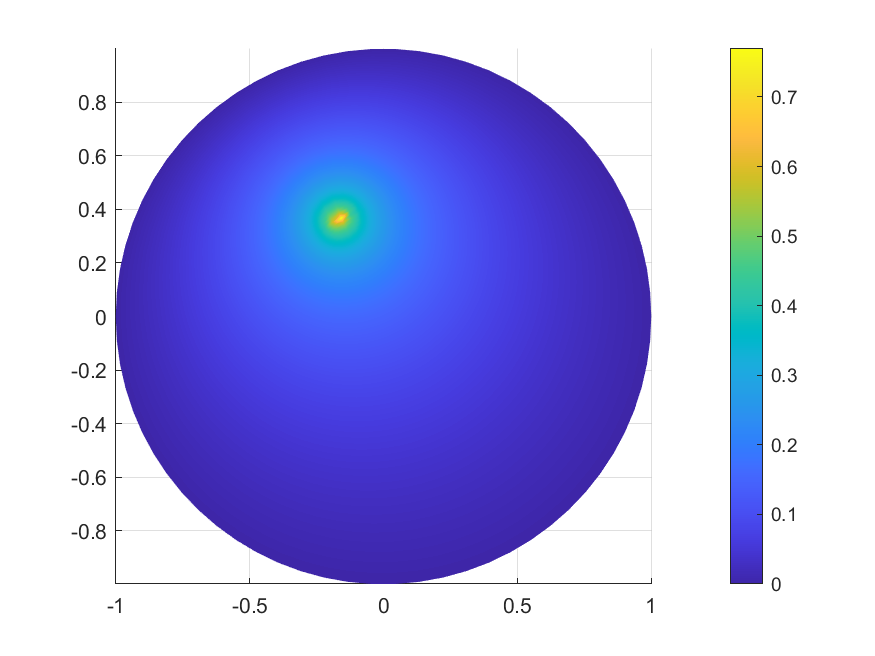}
}
\caption{The numerical solution $u_h$ of forward problem with $m=n=70$.}\label{F1}
\end{figure}

Following Algorithm \ref{A41} and the discussion in Remark \ref{R41},  we then take $m=n=60$ and $d=500$ to solve the inverse problem \eqref{eq-ibvp} using  the least-squares approach defined by equation \eqref{46}.
As shown in Figure \ref{F2}, when the noise level $\delta=3 \%,5 \%,10 \%$, our algorithm effectively estimates the point close to its actual position, demonstrating its robustness. The numerical results presented in Table  \ref{T1} further  validates this, showing that while the relative error increases with the noise level, the algorithm maintains a high level of accuracy and stability.

\begin{figure}
\centering
\subfigure[$3 \%$ random noise]{
\includegraphics[width=4.5cm]{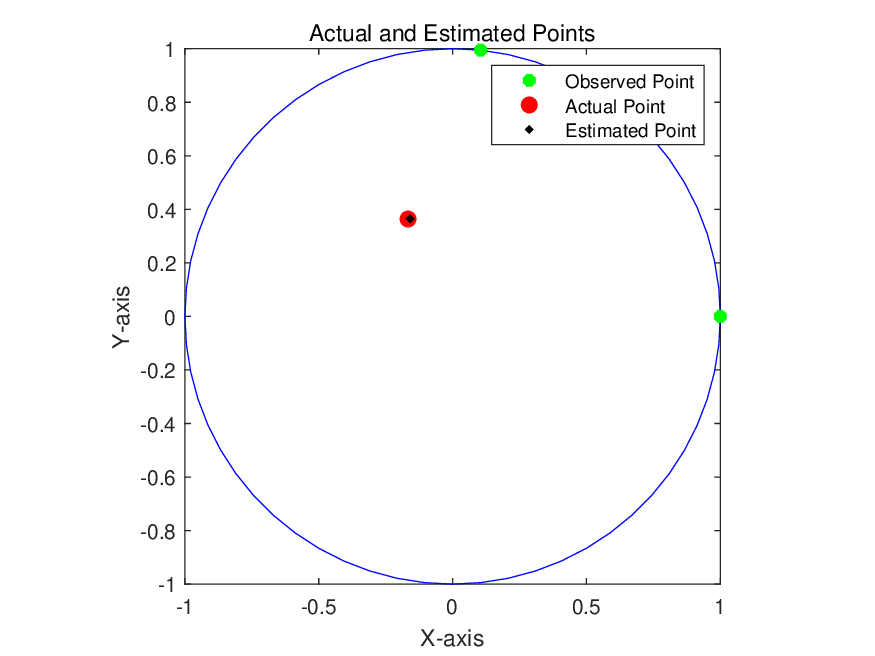}
}
\subfigure[$5 \%$ random noise]{
\includegraphics[width=4.5cm]{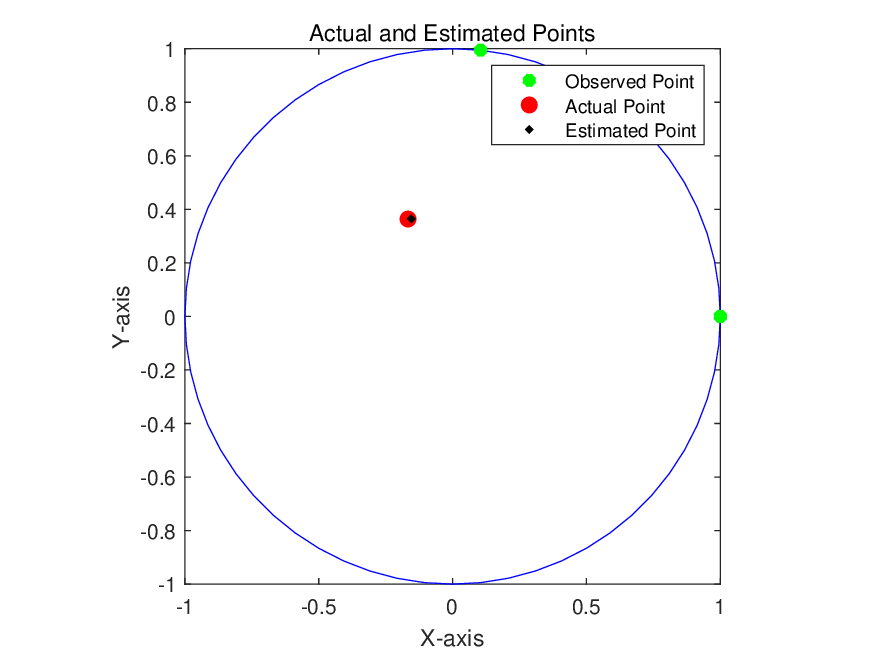}
}
\subfigure[$10 \%$ random noise]{
\includegraphics[width=4.5cm]{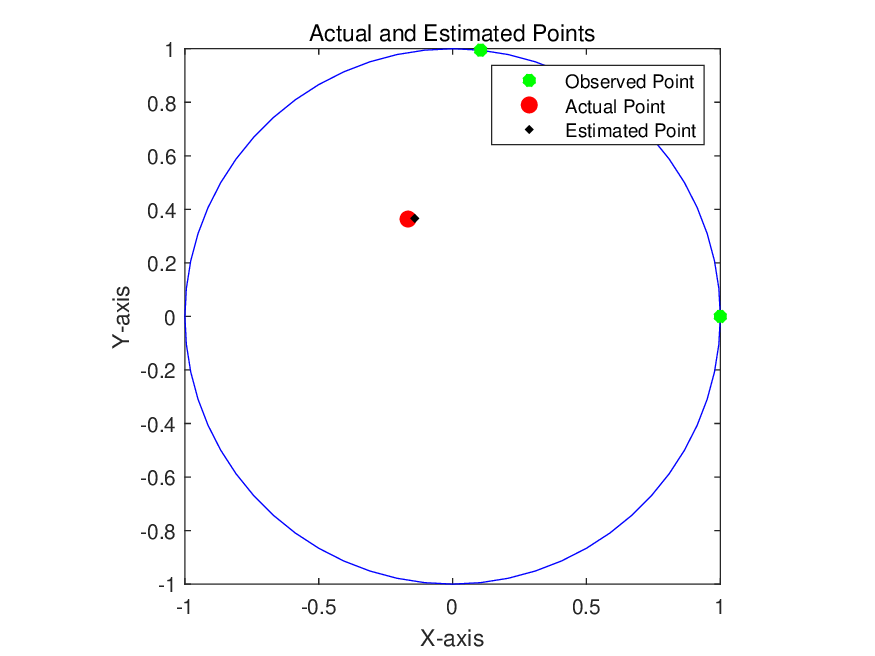}
}
\caption{ Observation, Actual and Estimated Points  \eqref{46}.}\label{F2}
\end{figure}

\begin{table}
\centering
\caption{Inverse computation of $r_1$ and $\theta_1$ with $3\%$, $5\%$, $10\%$ random noise \eqref{46}.}
\label{T1}
\begin{tabular}{ccccc}
\hline
           & $r_1$            & $\theta_1$           & $x_1$             & $y_1$             \\
\hline
Actual     & $0.4$            & $2.0$              & $-0.1665$         & $0.3637$          \\
\hline
Estimated ($\delta=3\%$)  & $0.3976$         & $1.9822$          & $-0.1590$         & $0.3644$          \\
Error       & $2.40 \times 10^{-3}$ & $1.78 \times 10^{-2}$  & $7.5 \times 10^{-3}$ & $7.05 \times 10^{-4}$ \\
\hline
Estimated ($\delta=5\%$)  & $0.3960$         & $1.9706$          & $-0.1541$         & $0.3648$          \\
Error       & $4.00 \times 10^{-3}$ & $2.94 \times 10^{-2}$  & $1.23 \times 10^{-2}$ & $1.05 \times 10^{-3}$ \\
\hline
Estimated ($\delta=10\%$) & $0.3926$         & $1.9390$          & $-0.1413$         & $0.3663$          \\
Error       & $7.40 \times 10^{-3}$ & $6.10 \times 10^{-2}$  & $2.52 \times 10^{-2}$ & $2.57 \times 10^{-3}$ \\
\hline
\end{tabular}
\end{table}

\begin{example}\label{exam3}\
In this example, we consider the observed points that are located away from the actual point by equation \eqref{46}.
\end{example}
As illustrated in Figure \ref{yF1}, for noise levels $\delta=3 \%,5 \%,10 \%$, our algorithm accurately estimates the point close to its true position, demonstrating its robustness. The observed points that are located away from the actual point are taken into consideration. The results shown in Table \ref{yT1} confirms this, revealing that even though the relative error grows with the noise level, the algorithm consistently maintains high accuracy and stability.

\begin{figure}
\centering
\subfigure[$3 \%$ random noise]{
\includegraphics[width=4.5cm]{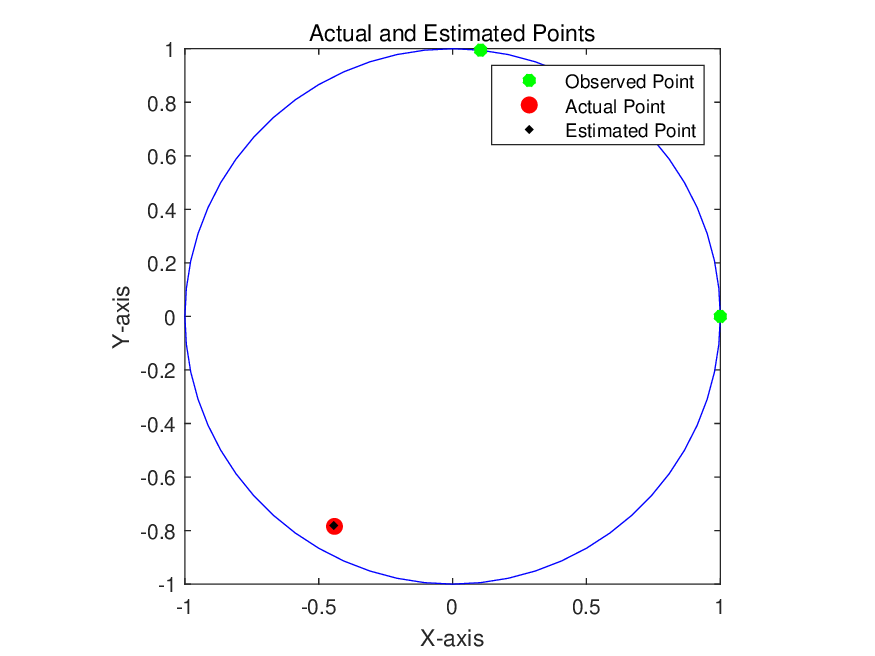}
}
\subfigure[$5 \%$ random noise]{
\includegraphics[width=4.5cm]{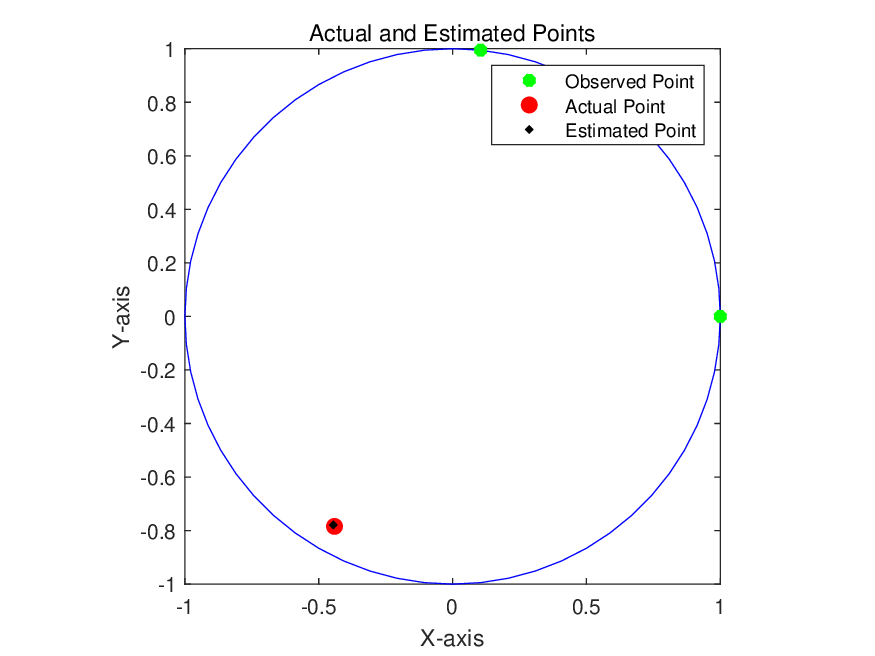}
}
\subfigure[$10 \%$ random noise]{
\includegraphics[width=4.5cm]{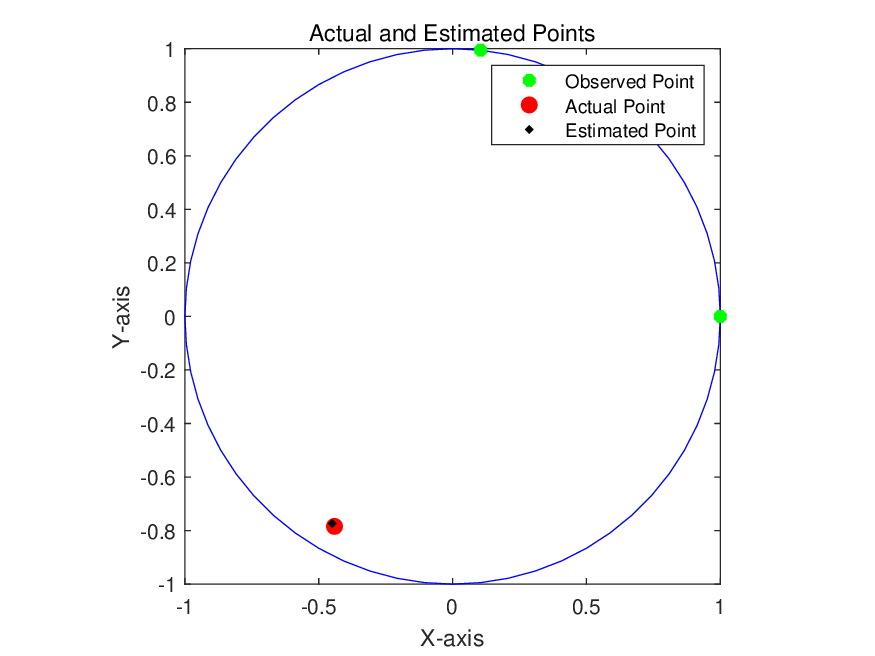}
}
\caption{ Observation, Actual and Estimated Points  \eqref{46}.}\label{yF1}
\end{figure}

\begin{table}
\centering
\caption{Inverse computation of $r_1$ and $\theta_1$ with $3\%$, $5\%$, $10\%$ random noise \eqref{46}.}
\label{yT1}
\begin{tabular}{ccccc}
\hline
           & $r_1$            & $\theta_1$           & $x_1$             & $y_1$             \\
\hline
Actual     & $0.9$            & $4.2$              & $-0.4412$         & $-0.7844$         \\
\hline
Estimated ($\delta=3\%$)  & $0.8984$        & $4.1956$          & $-0.4439$         & $-0.7811$         \\
Error       & $1.60 \times 10^{-3}$ & $4.40 \times 10^{-3}$             & $2.7 \times 10^{-3}$ & $3.3\times 10^{-3}$ \\
\hline
Estimated ($\delta=5\%$)  & $0.8974$        & $4.1925$          & $-0.4458$         & $-0.7788$         \\
Error       & $2.6 \times 10^{-3}$ & $7.50 \times 10^{-3}$ & $4.6 \times 10^{-3}$ & $5.6 \times 10^{-3}$ \\
\hline
Estimated ($\delta=10\%$) & $0.8948$        & $4.1857$          & $-0.4498$         & $-0.7735$         \\
Error       & $5.2 \times 10^{-3}$ & $1.43 \times 10^{-2}$ & $8.6 \times 10^{-3}$ & $1.09 \times 10^{-2}$ \\
\hline
\end{tabular}
\end{table}

\begin{example}\label{exam10}\
In this example, we consider the observed points that are located close to the actual point by equation \eqref{46}.
\end{example}
Figure \ref{j1s1} demonstrates that with noise levels $\delta=3 \%,5 \%,10 \%$,
 the algorithm is able to estimate the point close to its actual position, showing its resilience. The observed points, which are located close to the actual point, are also considered in this process. The data in Table \ref{Tj1s1} further support this observation, indicating that although the relative error increases with the noise level, the algorithm's accuracy and stability remain at a high level.

\begin{figure}
\centering
\subfigure[$3 \%$ random noise]{
\includegraphics[width=4.5cm]{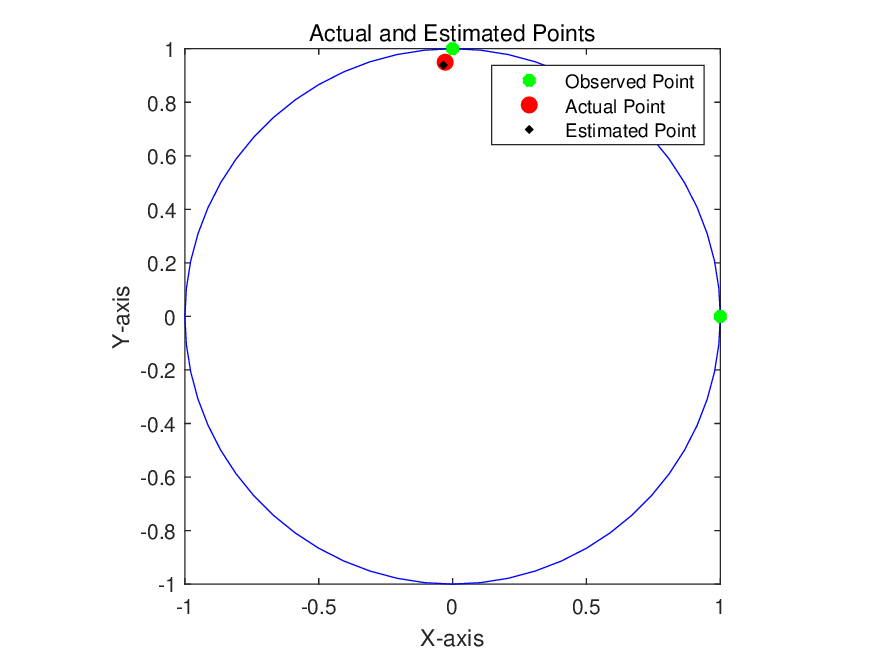}
}
\subfigure[$5 \%$ random noise]{
\includegraphics[width=4.5cm]{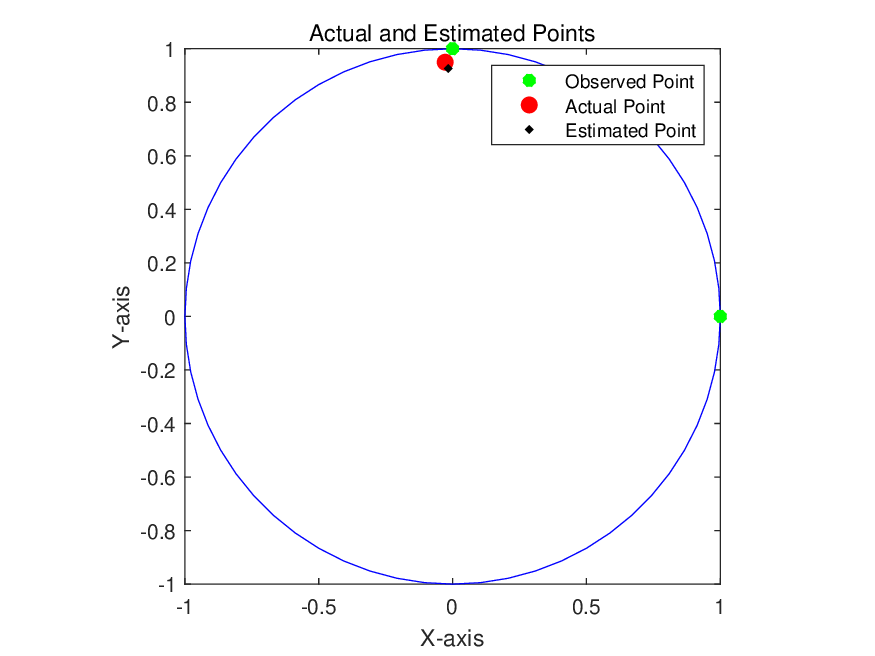}
}
\subfigure[$10 \%$ random noise]{
\includegraphics[width=4.5cm]{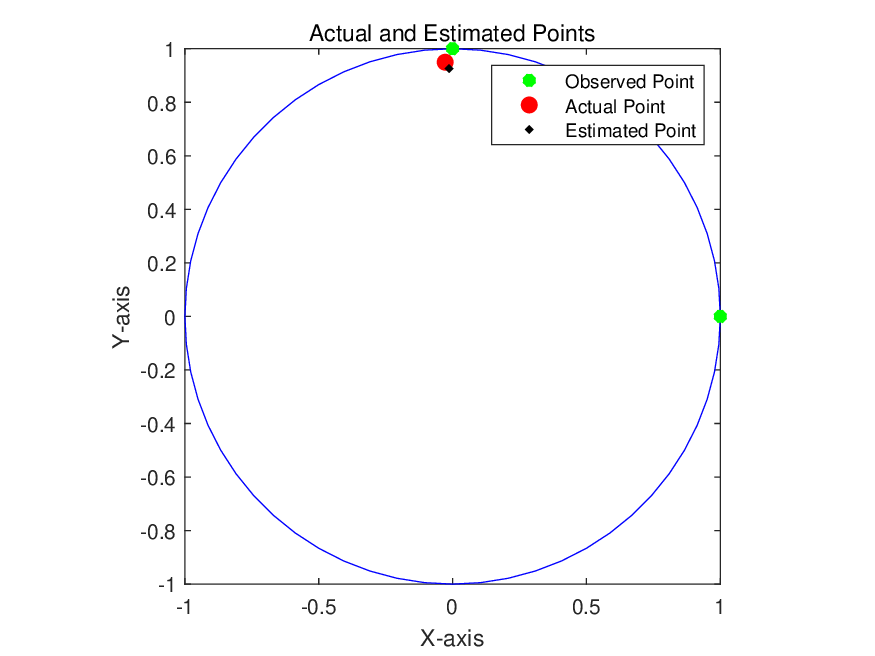}
}
\caption{ Observation, Actual and Estimated Points  \eqref{46}.}\label{j1s1}
\end{figure}

\begin{table}
\centering
\caption{Inverse computation of $r_1$ and $\theta_1$ with $3\%$, $5\%$, $10\%$ random noise \eqref{46}.}
\label{Tj1s1}
\begin{tabular}{ccccc}
\hline
           & $r_1$            & $\theta_1$           & $x_1$             & $y_1$             \\
\hline
Actual     & $0.95$           & $1.6$              & $-0.0277$         & $0.9496$          \\
\hline
Estimated ($\delta=3\%$)  & $0.9396$         & $1.6075$          & $-0.0345$         & $ 0.9390$          \\
Error       & $1.04 \times 10^{-2}$ & $7.50 \times 10^{-3}$  & $6.70 \times 10^{-3}$ & $1.06 \times 10^{-2}$ \\
\hline
Estimated ($\delta=5\%$)  & $0.9260$         & $1.5884$          & $-0.0163$         & $0.9259$          \\
Error       & $2.40 \times 10^{-2}$ & $1.16 \times 10^{-2}$  & $1.14 \times 10^{-2}$ & $2.37 \times 10^{-2}$ \\
\hline
Estimated ($\delta=10\%$) & $0.9260$         & $1.5849$          & $-0.0131$         & $0.9259$          \\
Error       & $2.40 \times 10^{-2}$ & $1.51 \times 10^{-2}$  & $1.47 \times 10^{-2}$ & $2.37 \times 10^{-2}$ \\
\hline
\end{tabular}
\end{table}

To support potential future extensions of our theory, we have designed three numerical schemes for the inversion of point sources.  The least-squares approach aims to minimize the objective function \( J \), defined as follows:

\paragraph{For two observed points in space and one point in time:}
\begin{equation}\label{51}
\begin{split}
&J_1(r_1,\theta_1)= \sum_{j=1}^{2}\left(\frac{\partial u_h}{\partial \overrightarrow{\mathbf{n}}}(z_j, t^*)-f_{ij}^{\delta}\right)^2, ~t^* \in(0, T),~ z_1, z_2 \in Z_{o b} \subset \partial \Omega.
\end{split}
\end{equation}

\paragraph{For one observed point in space and several points in time:}
\begin{equation}\label{52}
\begin{split}
&J_2(r_1,\theta_1)= \int_0^T\left(\frac{\partial u_h}{\partial \overrightarrow{\mathbf{n}}}(z^*, t)-f_j^{\delta}\right)^2\ dt, ~t \in(0, T),~ z^* \in Z_{o b} \subset \partial \Omega.
\end{split}
\end{equation}

\paragraph{For one observed point in space and two points in time:}
\begin{equation}\label{53}
\begin{split}
&J_3(r_1,\theta_1)= \sum_{i=1}^{2}\left(\frac{\partial u_h}{\partial \overrightarrow{\mathbf{n}}}(z^*, t_i)-f_{ij}^{\delta}\right)^2, ~t_1,~t_2 \in(0, T),~ z^* \in Z_{o b} \subset \partial \Omega.
\end{split}
\end{equation}

\begin{example}\label{exam12}\
 In this example, we consider the observed points that are located away
from the actual point.
\end{example}

\begin{itemize}
    \item We use two observed points in space and one point in time \eqref{51}.  See Figure \ref{Fs2t1y} and Table \ref{Ts2t1y} with $\delta=3 \%, 5 \%, 10 \%$.
    \item We use one observed point in space and several points in time \eqref{52}.  See Figure \ref{Fs1tsy} and Table \ref{Ts1tsy} with $\delta=3 \%, 5 \%, 10 \%$.
    \item We use one observed point in space and two points in time \eqref{53}.  See Figure \ref{Fs1t2y} and Table \ref{Ts1t2y} with $\delta=3 \%, 5 \%, 10 \%$.
\end{itemize}

From Figures \ref{Fs2t1y} to  \ref{Fs1t2y} and Tables \ref{Ts2t1y} to \ref{Ts1t2y}, it is evident that when the observed points deviate significantly from the actual point, the performance of algorithms \eqref{51}, \eqref{52}, and \eqref{53} deteriorates, and they exhibit poor robustness to noise.

\begin{figure}
\centering
\subfigure[$3 \%$ random noise]{
\includegraphics[width=4.5cm]{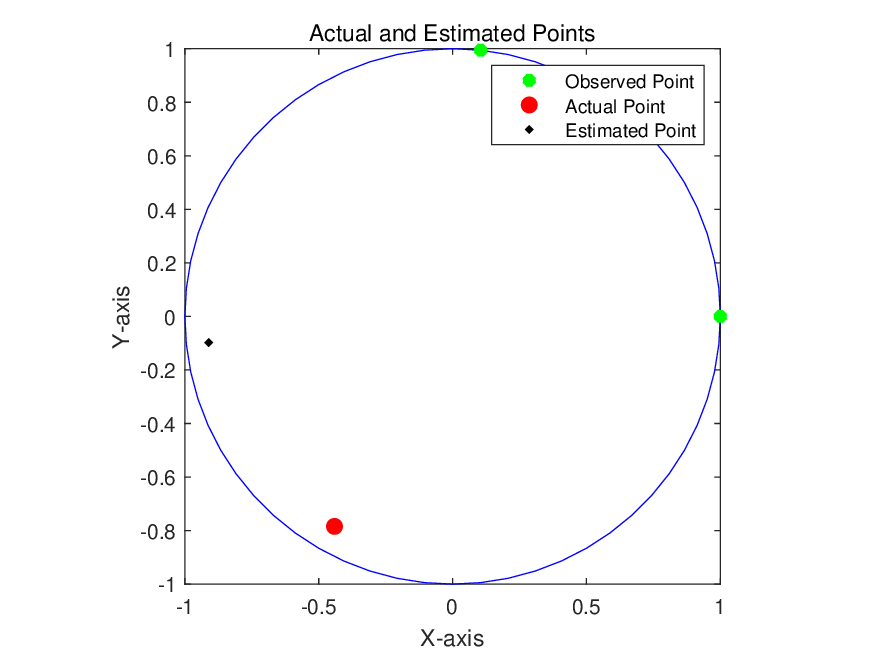}
}
\subfigure[$5 \%$ random noise]{
\includegraphics[width=4.5cm]{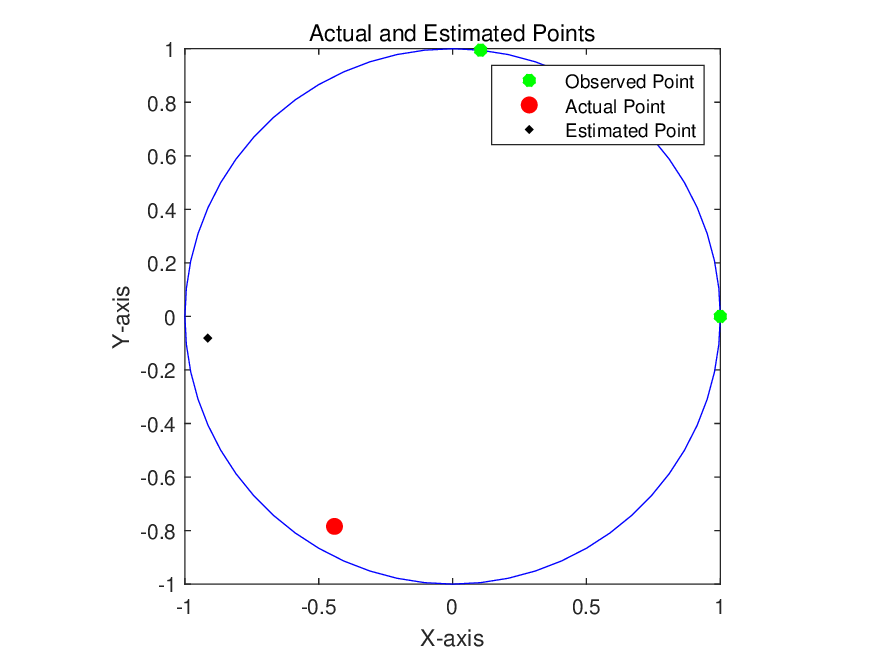}
}
\subfigure[$10 \%$ random noise]{
\includegraphics[width=4.5cm]{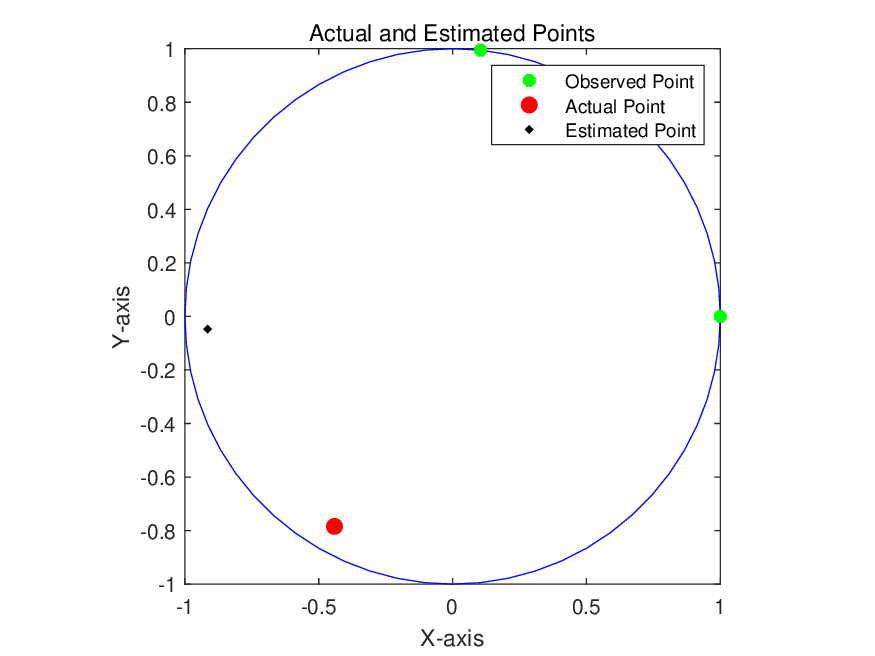}
}
\caption{ Observation, Actual and Estimated Points  \eqref{51}.}\label{Fs2t1y}
\end{figure}
\begin{table}
\centering
\caption{Inverse computation of $r_1$ and $\theta_1$ with $3\%$, $5\%$, $10\%$ random noise \eqref{51}.}\label{Ts2t1y}
\begin{tabular}{ccccc}
\hline
           & $r_1$            & $\theta_1$           & $x_1$             & $y_1$             \\
\hline
Actual     & $0.9$            & $4.2$              & $-0.4412$         & $-0.7844$         \\
\hline
Estimated ($\delta=3\%$)  & $0.9163$        & $3.2484$          & $-0.9111$         & $-0.0977$         \\
Error       & $1.63 \times 10^{-2}$ & $9.516 \times 10^{-1}$  & $4.698 \times 10^{-1}$ & $6.867 \times 10^{-1}$ \\
\hline
Estimated ($\delta=5\%$)  & $0.9182$        & $3.2300$          & $-0.9146$         & $-0.0811$         \\
Error       & $1.82 \times 10^{-2}$ & $9.7 \times 10^{-1}$  & $4.734 \times 10^{-1}$ & $7.033 \times 10^{-1}$ \\
\hline
Estimated ($\delta=10\%$) & $0.9162$        & $3.1934$          & $-0.9150$         & $-0.0474$         \\
Error       & $1.62 \times 10^{-2}$ & $1.0066 \times 10^{0}$  & $4.737 \times 10^{-1}$ & $7.370 \times 10^{-1}$ \\
\hline
\end{tabular}
\end{table}

\begin{figure}
\centering
\subfigure[$3 \%$ random noise]{
\includegraphics[width=4.5cm]{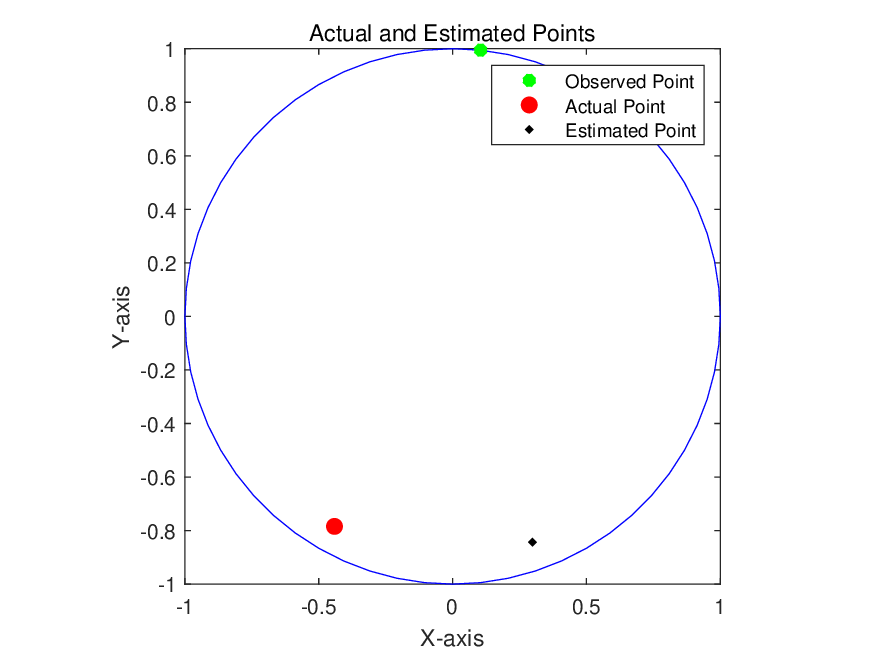}
}
\subfigure[$5 \%$ random noise]{
\includegraphics[width=4.5cm]{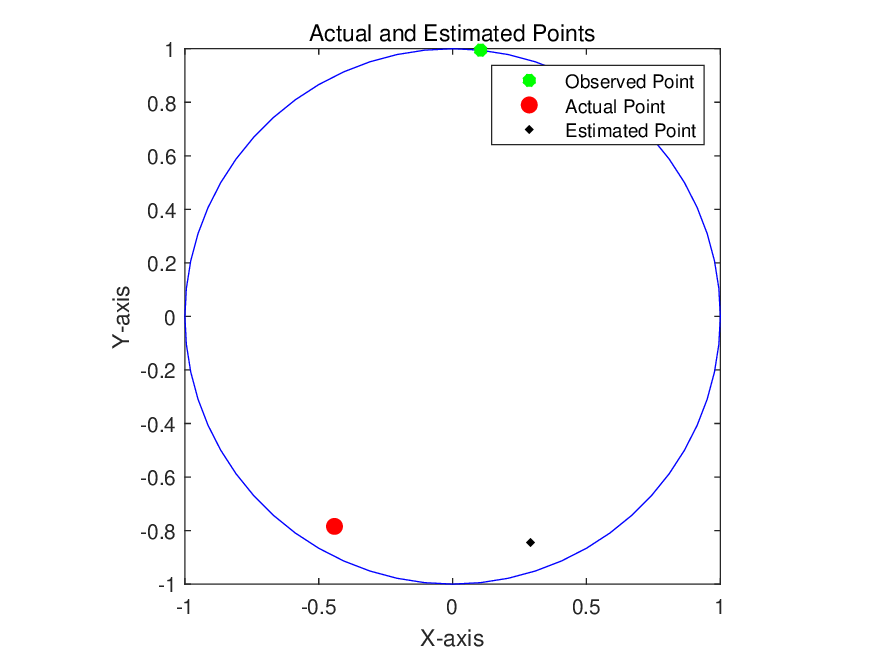}
}
\subfigure[$10 \%$ random noise]{
\includegraphics[width=4.5cm]{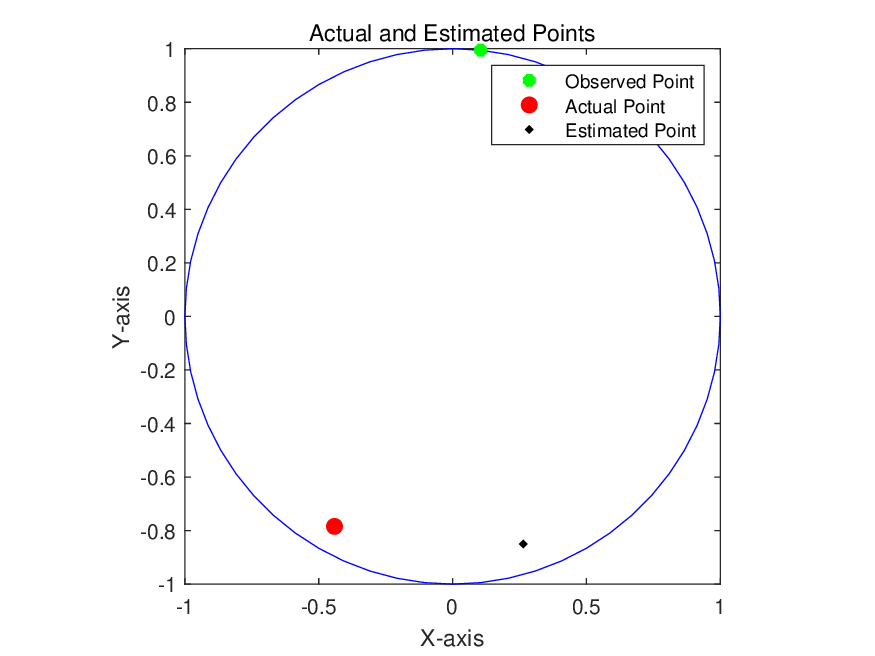}
}
\caption{ Observation, Actual and Estimated Points  \eqref{52}.}\label{Fs1tsy}
\end{figure}

\begin{table}
\centering
\caption{Inverse computation of $r_1$ and $\theta_1$ with $3\%$, $5\%$, $10\%$ random noise \eqref{52}.}\label{Ts1tsy}
\begin{tabular}{ccccc}
\hline
           & $r_1$            & $\theta_1$           & $x_1$             & $y_1$             \\
\hline
Actual     & $0.9$            & $4.2$              & $-0.4412$         & $-0.7844$         \\
\hline
Estimated ($\delta=3\%$)  & $0.8944$        & $5.0523$          & $-0.5622$         & $-0.7077$         \\
Error       & $5.6 \times 10^{-3}$ & $8.523 \times 10^{-1}$ & $7.394 \times 10^{-1}$ & $5.88 \times 10^{-2}$ \\
\hline
Estimated ($\delta=5\%$)  & $0.8932$        & $5.0443$          & $0.2910$         & $-0.8445$         \\
Error       & $6.8 \times 10^{-3}$ & $8.443 \times 10^{-1}$ & $7.323 \times 10^{-1}$ & $6.00 \times 10^{-2}$ \\
\hline
Estimated ($\delta=10\%$) & $0.8901$        & $5.0133$          & $ 0.2638$         & $-0.8501$         \\
Error       & $9.9 \times 10^{-3}$ & $8.133 \times 10^{-1}$  & $7.051 \times 10^{-1}$ & $6.57 \times 10^{-2}$ \\
\hline
\end{tabular}
\end{table}

\begin{figure}
\centering
\subfigure[$3 \%$ random noise]{
\includegraphics[width=4.5cm]{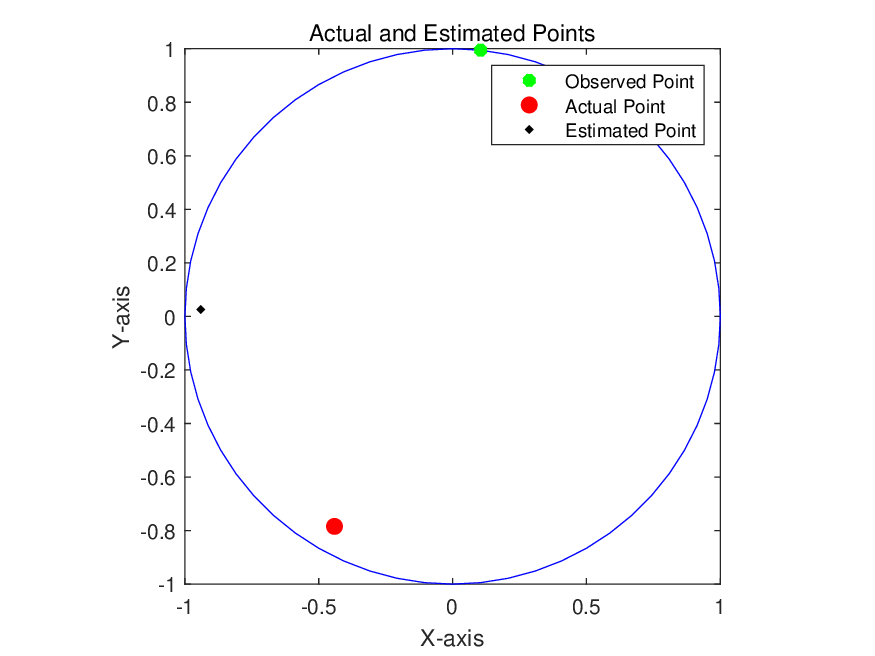}
}
\subfigure[$5 \%$ random noise]{
\includegraphics[width=4.5cm]{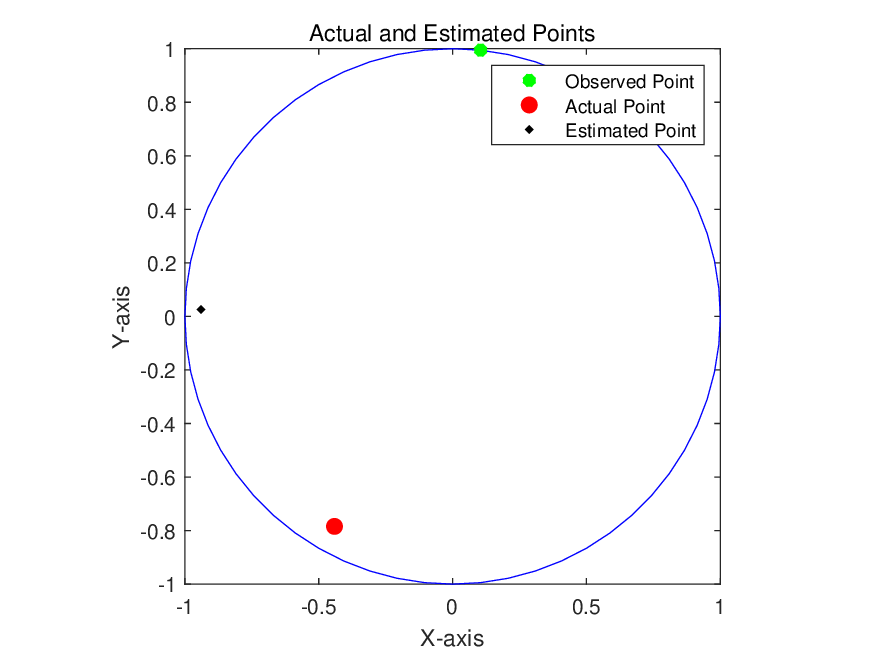}
}
\subfigure[$10 \%$ random noise]{
\includegraphics[width=4.5cm]{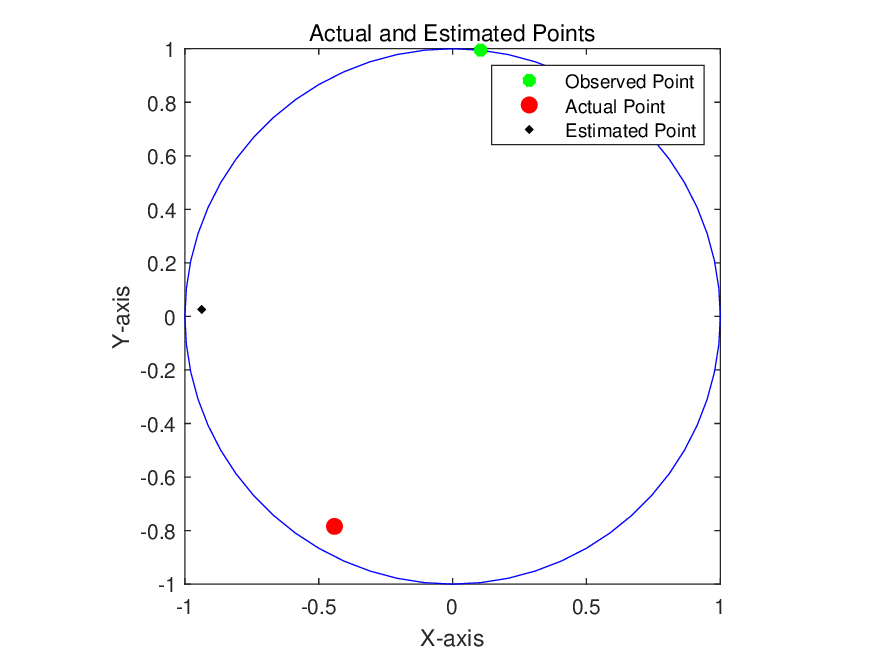}
}
\caption{ Observation, Actual and Estimated Points  \eqref{53}.}\label{Fs1t2y}
\end{figure}

\begin{table}
\centering
\caption{Inverse computation of $r_1$ and $\theta_1$ with $3\%$, $5\%$, $10\%$ random noise \eqref{53}.}
\label{Ts1t2y}
\begin{tabular}{ccccc}
\hline
           & $r_1$ & $\theta_1$ & $x_1$ & $y_1$ \\
\hline
Actual                    & $0.9$       & $4.2$       & $-0.4412$   & $-0.7844$   \\
\hline
Estimated ($\delta=3\%$)  & $0.9412$    & $3.1141$    & $-0.9408$   & $0.0259$   \\
Error                     & $4.12 \times 10^{-2}$ & $1.0859 \times 10^{0}$ & $4.966 \times 10^{-1}$ & $8.103 \times 10^{-1}$ \\
\hline
Estimated ($\delta=5\%$)  & $0.9402$   & $3.1139$    & $-0.9398$   & $ 0.0260$   \\
Error                     & $4.02 \times 10^{-2}$ & $1.0861 \times 10^{0}$ & $4.986 \times 10^{-1}$ & $8.105 \times 10^{-1}$ \\
\hline
Estimated ($\delta=10\%$) & $0.9375$   & $3.1137$    & $-0.9371$   & $-0.0261$   \\
Error                     & $3.75 \times 10^{-2}$ & $1.0863 \times 10^{0}$ & $4.959 \times 10^{-1}$ & $8.106\times 10^{-1}$ \\
\hline
\end{tabular}
\end{table}

\begin{example}\label{exam13}\
In this example, we consider the observed points that are located  close to
 the actual point.
\end{example}

\begin{itemize}
    \item We use two observed points in space and one point in time \eqref{51}.  See Figure \ref{Fs2t1j} and  Table \ref{Ts2t1j} with  $\delta=3 \%, 5 \%, 10 \%$.
    \item We use one observed point in space and several points in time \eqref{52}. See Figure \ref{Fs1tsj} and  Table \ref{Ts1tsj} with  $\delta=3 \%, 5 \%, 10 \%$.
    \item We use one observed point in space and two points in time \eqref{53}.  See Figure \ref{Fs1t2j} and  Table \ref{Ts1t2j} with  $\delta=3 \%, 5 \%, 10 \%$.
\end{itemize}

As shown in Figures \ref{Fs2t1j} to  \ref{Fs1t2j} and Tables \ref{Ts2t1j} to \ref{Ts1t2j}, algorithms \eqref{51}, \eqref{52}, and \eqref{53} perform well and exhibit strong robustness to noise when the observed points are close to the actual point.

\begin{figure}
\centering
\subfigure[$3 \%$ random noise]{
\includegraphics[width=4.5cm]{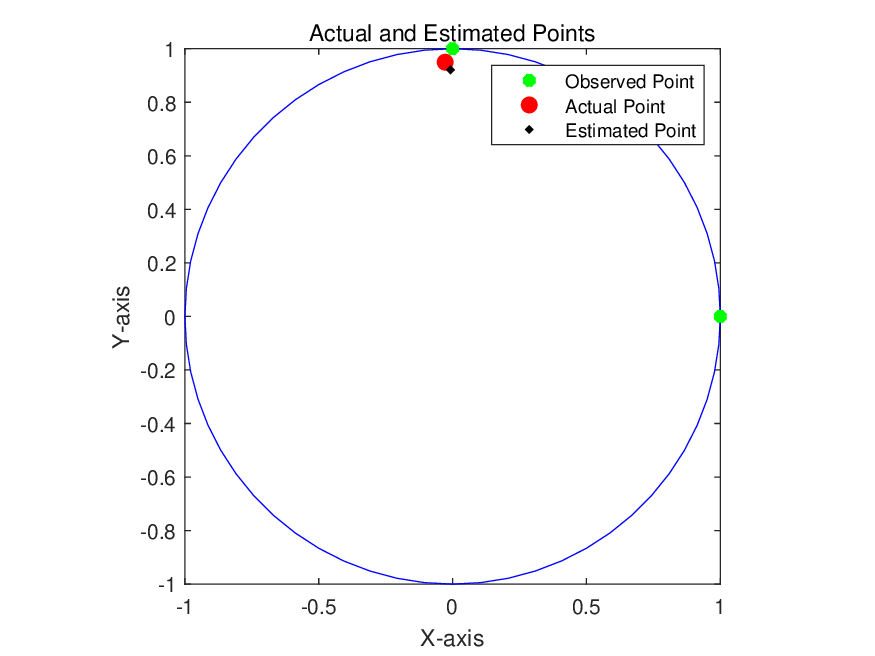}
}
\subfigure[$5 \%$ random noise]{
\includegraphics[width=4.5cm]{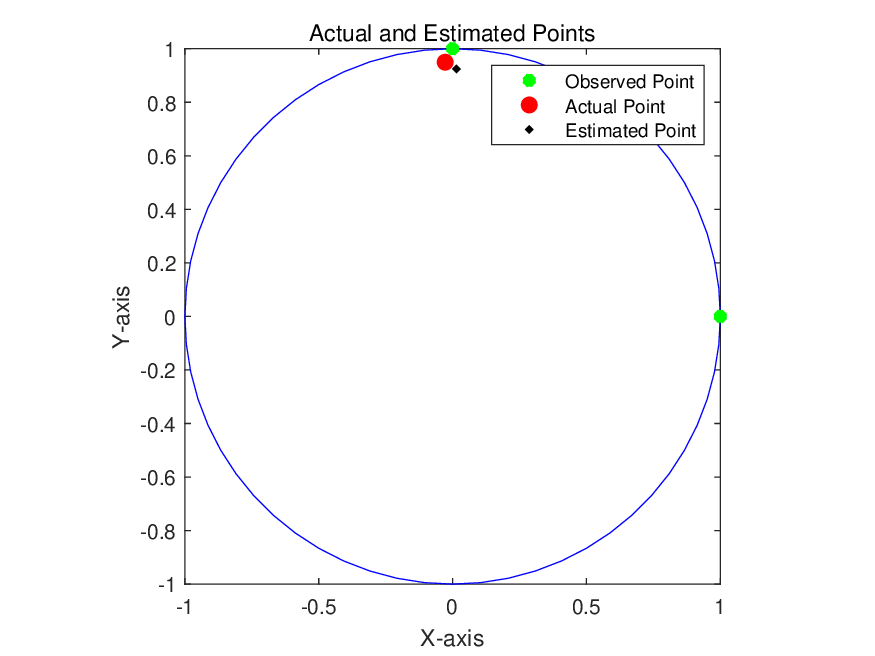}
}
\subfigure[$10 \%$ random noise]{
\includegraphics[width=4.5cm]{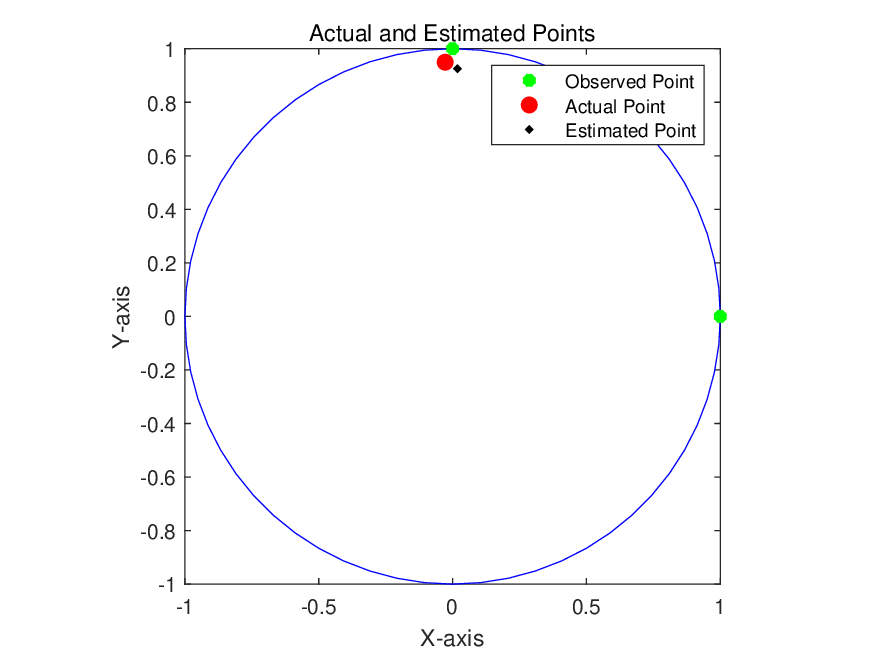}
}
\caption{ Observation, Actual and Estimated Points  \eqref{51}.}\label{Fs2t1j}
\end{figure}

\begin{table}
\centering
\caption{Inverse computation of $r_1$ and $\theta_1$ with $3\%$, $5\%$, $10\%$ random noise \eqref{51}.}
\label{Ts2t1j}
\begin{tabular}{ccccc}
\hline
           & $r_1$ & $\theta_1$ & $x_1$ & $y_1$ \\
\hline
Actual                    & $0.95$      & $1.6$       & $-0.0277$   & $0.9496$   \\
\hline
Estimated ($\delta=3\%$)  & $0.9209$   & $1.5794$    & $-0.0079$    & $0.9209$   \\
Error                     & $2.91 \times 10^{-2}$ & $2.06 \times 10^{-2}$ & $1.98 \times 10^{-2}$ & $2.87 \times 10^{-2}$ \\
\hline
Estimated ($\delta=5\%$)  & $0.9245$   & $1.5553$    & $0.0143$    & $0.9244$   \\
Error                     & $2.55 \times 10^{-2}$ & $4.47 \times 10^{-2}$ & $4.21 \times 10^{-2}$ & $2.52 \times 10^{-2}$ \\
\hline
Estimated ($\delta=10\%$) & $0.9256$   & $1.5512$    & $0.0181$   & $0.9254$   \\
Error                     & $2.44 \times 10^{-2}$ & $4.88 \times 10^{-2}$ & $4.59 \times 10^{-2}$ & $2.42 \times 10^{-2}$ \\
\hline
\end{tabular}
\end{table}

\begin{figure}
\centering
\subfigure[$3 \%$ random noise]{
\includegraphics[width=4.5cm]{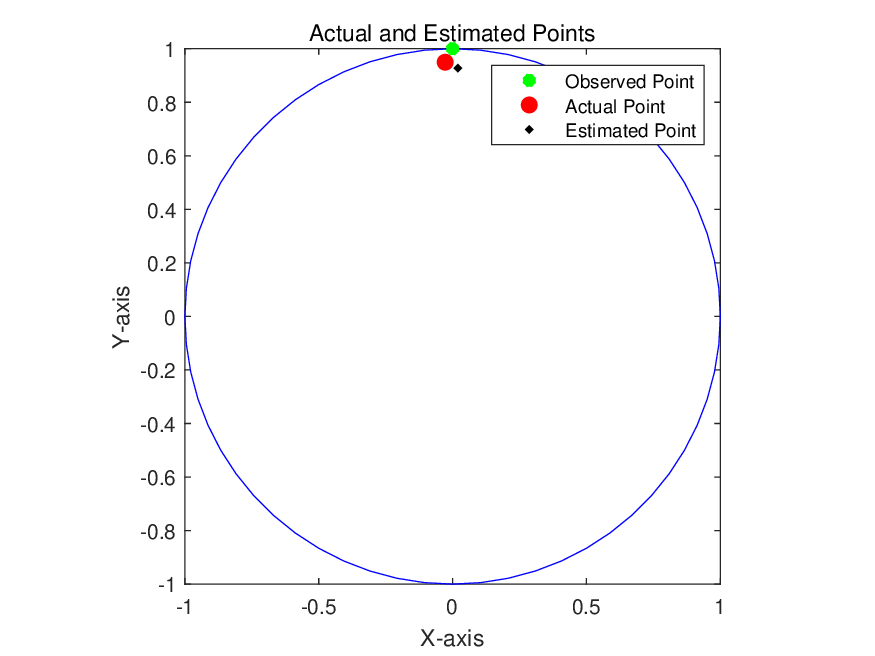}
}
\subfigure[$5 \%$ random noise]{
\includegraphics[width=4.5cm]{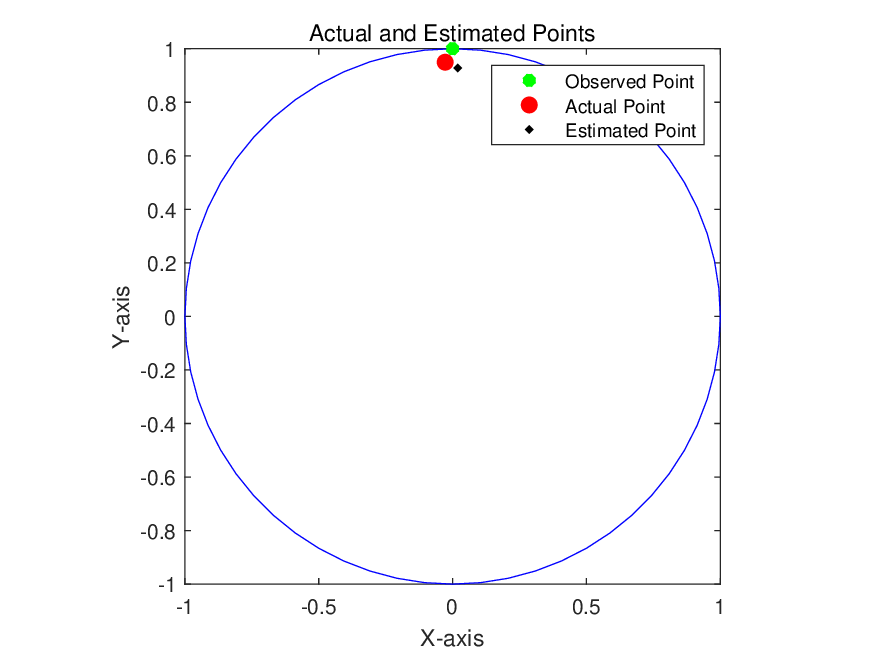}
}
\subfigure[$10 \%$ random noise]{
\includegraphics[width=4.5cm]{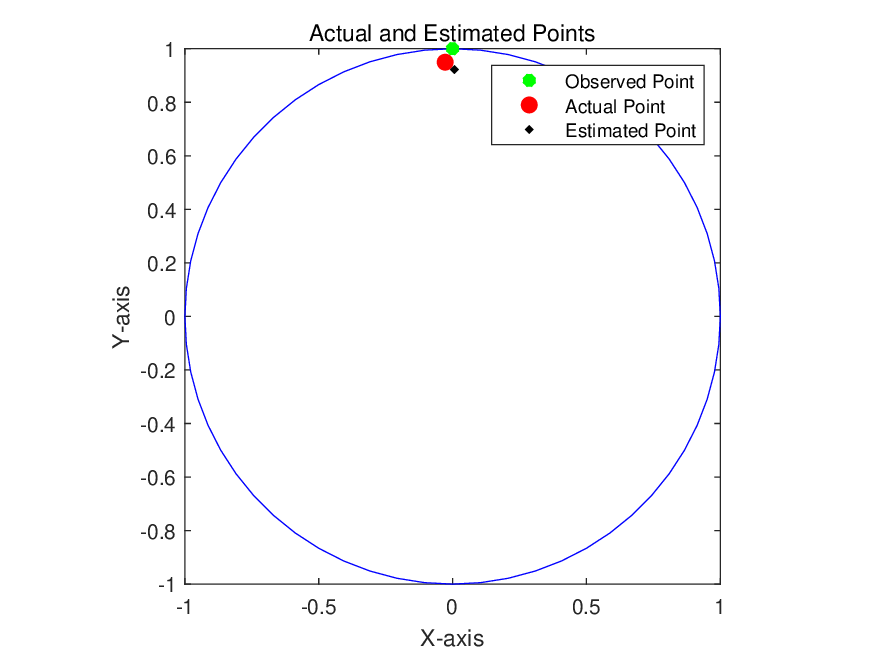}
}
\caption{ Observation, Actual and Estimated Points  \eqref{52}.}\label{Fs1tsj}
\end{figure}

\begin{table}
\centering
\caption{Inverse computation of $r_1$ and $\theta_1$ with $3\%$, $5\%$, $10\%$ random noise \eqref{52}.}\label{Ts1tsj}
\begin{tabular}{ccccc}
\hline
           & $r_1$            & $\theta_1$           & $x_1$             & $y_1$             \\
\hline
Actual     & $0.95$           & $1.6$              & $-0.0277$         & $0.9496$          \\
\hline
Estimated ($\delta=3\%$)  & $0.9277$        & $1.5494$          & $0.0198$          & $ 0.9275$          \\
Error       & $2.23 \times 10^{-2}$ & $5.06 \times 10^{-2}$ & $4.76 \times 10^{-2}$ & $2.21 \times 10^{-2}$ \\
\hline
Estimated ($\delta=5\%$)  & $0.9280$        & $1.5502$          & $0.0191$          & $0.9278$          \\
Error       & $2.20 \times 10^{-2}$ & $4.98 \times 10^{-2}$ & $4.69 \times 10^{-2}$ & $2.18 \times 10^{-2}$ \\
\hline
Estimated ($\delta=10\%$) & $0.9222$        & $1.5634$          & $0.0068$          & $0.9222$          \\
Error       & $2.78 \times 10^{-2}$ & $3.66 \times 10^{-2}$ & $3.46 \times 10^{-2}$ & $2.74 \times 10^{-2}$ \\
\hline
\end{tabular}
\end{table}

\begin{table}
\centering
\caption{Inverse computation of $r_1$ and $\theta_1$ with $3\%$, $5\%$, $10\%$ random noise \eqref{53}.}\label{Ts1t2j}
\begin{tabular}{ccccc}
\hline
           & $r_1$            & $\theta_1$           & $x_1$             & $y_1$             \\
\hline
Actual     & $0.95$           & $1.6$              & $-0.0277$         & $0.9496$          \\
\hline
Estimated ($\delta=3\%$)  & $0.9231$         & $1.5854$          & $-0.0135$          & $0.9230$          \\
Error       & $2.69 \times 10^{-2}$ & $1.46 \times 10^{-2}$ & $1.43 \times 10^{-2}$ & $2.66 \times 10^{-2}$ \\
\hline
Estimated ($\delta=5\%$)  & $0.9198$        & $1.5631$          & $0.0071$          & $0.9198$          \\
Error       & $3.02 \times 10^{-2}$ & $3.69 \times 10^{-2}$ & $3.48 \times 10^{-2}$ & $2.98 \times 10^{-2}$ \\
\hline
Estimated ($\delta=10\%$) & $0.9244$        & $1.5840$          & $-0.0122$          & $0.9243$          \\
Error       & $2.56 \times 10^{-2}$ & $1.60 \times 10^{-2}$ & $1.55 \times 10^{-2}$ & $2.53 \times 10^{-2}$ \\
\hline
\end{tabular}
\end{table}

\begin{figure}
\centering
\subfigure[$3 \%$ random noise]{
\includegraphics[width=4.5cm]{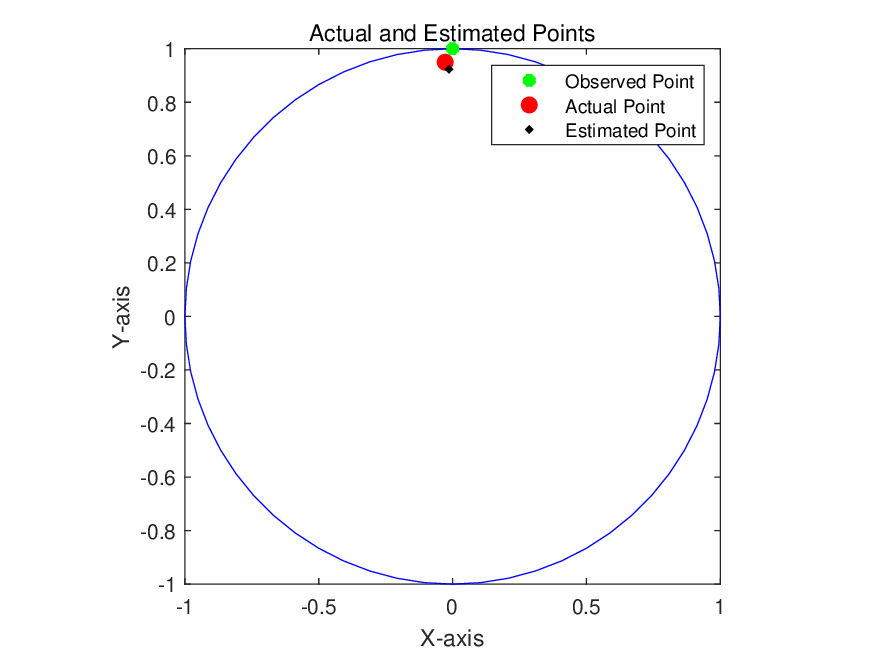}
}
\subfigure[$5 \%$ random noise]{
\includegraphics[width=4.5cm]{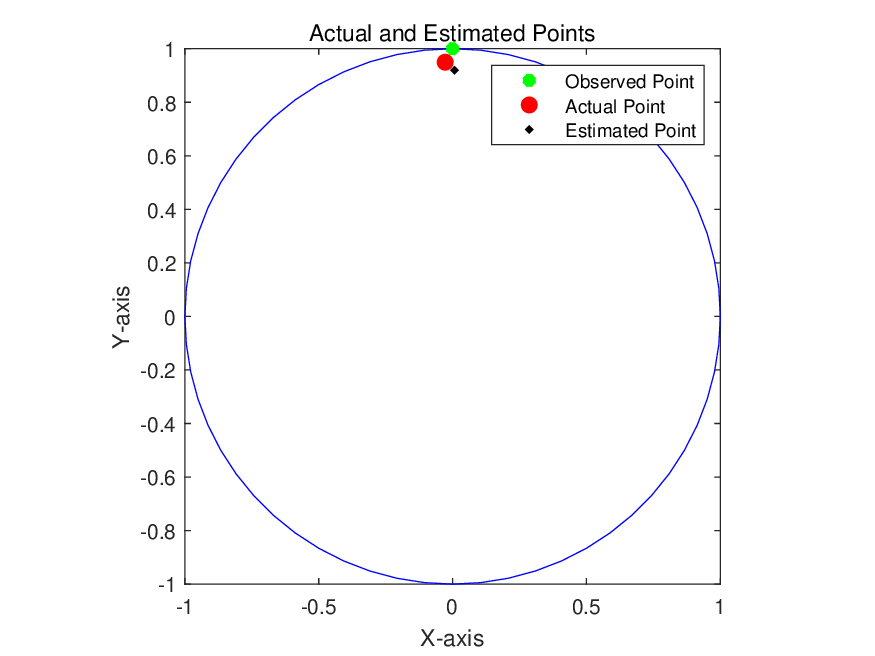}
}
\subfigure[$10 \%$ random noise]{
\includegraphics[width=4.5cm]{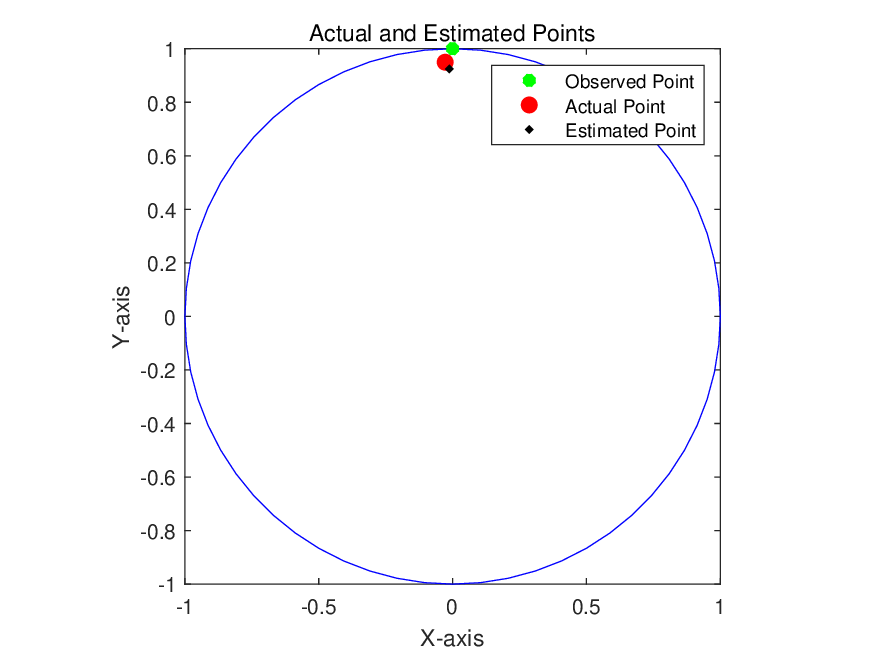}
}
\caption{ Observation, Actual and Estimated Points  \eqref{53}.}\label{Fs1t2j}
\end{figure}

\section{Acknowledgments.}
Zhidong Zhang is supported by the National Key Research and Development Plan of China(Grant No.2023YFB3002400), the National Natural Science Foundation of China (Grant No. 12101627), and the Fundamental Research Funds for the Central Universities, Sun Yat-sen University (Grant No. 24qnpy304). Wenlong Zhang is partially supported by the National Natural Science Foundation of China under grant numbers No.12371423 and No.12241104.

\section{Appendix.}
  \begin{lemma}\label{lemma_smooth}
  For the model \eqref{eq-ibvp}, the boundary flux data $\frac{\partial u}{\n}(z,t)$ is smooth on $\partial\Omega\times (0,T)$.
\end{lemma}  
  \begin{proof}

First, we define the fundamental solution:  
\begin{equation} 
\Phi(x, t) =
\begin{cases} 
\frac{1}{4\pi t} e^{-\frac{|x|^2}{4t}}, & t > 0, \\ 
0, & t \leq 0,
\end{cases}
\end{equation}
which satisfies:  
\begin{equation}  
\left\{
\begin{aligned}
(\partial_t - \Delta) \Phi(x, t) &= 0, && (x, t) \in \Omega \times (0, T), \\
\Phi(x, 0) &= \delta_{s_*}.
\end{aligned}
\right.
\end{equation}

Next, we construct the Green's function \(G(x, s_*, t)\), which satisfies:  
\begin{equation}  
\left\{
\begin{aligned}
(\partial_t - \Delta) G(x, t) &= \delta_{s_*}, && (x, t) \in \Omega \times (0, T), \\
G(x, 0) &= 0, && (x, t) \in \Omega \times \{0\} \cup \partial\Omega \times (0, T).
\end{aligned}
\right.
\end{equation}

We then define:  
\[
\varphi_{s_*}(x, t) = \int_{0}^{t} \int_{\mathbb{R}^n} \Phi(x - y, t - s) \delta_{s_*} \, dy \, ds = \int_{0}^{t} \Phi(x - s_*, t - s) \, ds,
\]
where \(\varphi_{s_*}(x, t)\) satisfies:  
\begin{equation}  
\left\{
\begin{aligned}
(\partial_t - \Delta) \varphi(x, t) &= \delta_{s_*}, && (x, t) \in \mathbb{R}^2 \times (0, T), \\
\varphi(x, 0) &= 0.
\end{aligned}
\right.
\end{equation}

Now, define \(x_{s_*}^{\ast} = \frac{x_{s_*}}{|x_{s_*}|^2} \notin \Omega\). For \(x \in \partial \Omega\), we have:

\[
|x - x_{s_*}|^2 = |x_{s_*}|^2 |x - x_{s_*}^*|.
\]

From this, we observe that \(\Phi\) depends only on \(|x - x_{s_*}|\), and the same holds for \(\varphi\). Thus, we have:

\begin{equation}
\begin{aligned}
\Phi\left(x - x_{s_*}, t\right) 
&= \Phi\left( |x_{s_*}| \left( x - x_{s_*}^* \right), t \right), \quad x \in \partial \Omega\\
&= \frac{1}{4 \pi t} \exp \left( - \frac{ |x_{s_*}|^2 |x - x_{s_*}^*| }{ 4 t } \right), \quad x \in \partial \Omega.
\end{aligned}
\end{equation}

Let
\begin{equation}
\tilde{\varphi}(x, t) = \int_0^t \Phi\left(\left|x_{s_*}\right| \cdot \left(x - x_{s_*}^*\right), t-s\right) \, ds.
\end{equation}
Then, we have
\begin{equation}
\varphi(x, t) = \tilde{\varphi}\left(x, t\right), \quad \text{for } x \in \partial \Omega,
\end{equation}
and
\begin{equation}
\begin{aligned}
\tilde{\varphi}(x, t) &= \int_0^t \frac{1}{4 \pi \left(t-s\right)} 
\exp\left( -\frac{\left|x_{s_*}\right|^2 \cdot \left|x-x_{s_*}^*\right|}{4 \left(t-s\right)} \right) \, ds \\
&= \int_0^t \frac{1}{4 \pi \left( \frac{t-s}{\left|x_{s_*}\right|^2} \right)} 
\exp\left( -\frac{\left|x-x_{s_*}^*\right|}{4 \left( \frac{t-s}{\left|x_{s_*}\right|^2} \right)} \right)
\frac{1}{\left|x_{s_*}\right|^2} \, ds \\
&= \int_0^{\frac{t}{\left|x_{s_*}\right|^2}} 
\frac{1}{4 \pi \left( \frac{t}{\left|x_{s_*}\right|^2} - s \right)} 
\exp\left( -\frac{\left|x-x_{s_*}^*\right|}{4 \left( \frac{t}{\left|x_{s_*}\right|^2} - s \right)} \right) \, ds \\
&= \varphi_{x_{s_*}^*} \left( x, \frac{t}{\left|x_{s_*}\right|^2} \right) \\
&= \varphi_{x_{s_*}^*} \left( x, t' \right), \quad t' = \frac{t}{\left|x_{s_*}\right|^2}.
\end{aligned}
\end{equation}

For $\varphi_{x_{s_*}^*}$, it satisfies
 \begin{equation}
\left\{\begin{array}{l}
\varphi_{x_{s_*}^{*}t^{'}}-\Delta \varphi=\delta_{x_{s_*}^{*}}, \quad (x,t)\in R^2 \times(0, T), \\
\varphi\left(t^{\prime}=0\right)=0,
\end{array}\right.
\end{equation}
and $(\varphi_{x_{s_*}^{*}})_t=(\varphi_{x_{s_*}^{*}})_{t^{'}}\frac{1}{|x_{s_*}|^2}$, so we have $(\varphi_{x_{s_*}^{*}})_{t^{'}}=|x_{s_*}|^2(\varphi_{x_{s_*}^{*}})_t$.

In \( \Omega\), the following holds:  
\[
\left(\varphi_{x_{s_*}^*}\right)_{t'} - \Delta \varphi_{x_{s_*}^*} = 0, \quad \text{where } x_{s_*}^* \notin \Omega.
\]  
This implies:  
\[
\left|x_{s_*}\right|^2 \left(\varphi_{x_{s_*}^*}\right)_t = \Delta \varphi_{x_{s_*}^*}.
\]  
Equivalently,  
\[
\left(\varphi_{x_{s_*}^*}\right)_{t} - \Delta \varphi_{x_{s_*}^*} = \left(\frac{1}{\left|x_{s_*}\right|^2} - 1\right) \Delta \varphi_{x_{s_*}^*}, \quad \text{in } \Omega.
\]

Let $v=\varphi-\tilde{\varphi}$ and $v$ satisfies
\begin{equation}
\left\{\begin{array}{l}
v_t-\Delta v=\delta_{x_{s_*}}-\left(\frac{1}{\left|x_{s_*}\right|^2}-1\right) \Delta \varphi_{x_{s_*}^*}, ~~\text { in } \Omega, \\
v \mid _{\partial \Omega}=0, \\
v(t=0)=0.
\end{array}\right.
\end{equation}

Let $w$ satisfies 
\begin{equation}
\left\{\begin{array}{l}
w_t-\Delta w=-\left(\frac{1}{\left|x_{s_*}\right|^2}-1\right) \Delta \varphi_{x_{s_*}^*}~, ~ \text { in }  ~\Omega,\\
w \mid_{\partial \Omega}=0,  \\
w(t=0)=0.
\end{array}\right.
\end{equation}
Because $u=v-w$, and $v$ and $w$ are smooth function on boundary, we can prove $\frac{\partial u}{\n}(z,t)$ in \eqref{eq-ibvp} is smooth.
  \end{proof}

\bibliographystyle{plainurl} 
\bibliography{point_source}

\begin{thebibliography}{10}

\bibitem{Andrle2011}
M.~Andrle, F.~Ben~Belgacem, and A.~El~Badia.
\newblock Identification of moving pointwise sources in an
  advection-dispersion-reaction equation.
\newblock {\em Inverse Problems}, 27(2):025007 (21pp), 2011.
\newblock \href {https://doi.org/10.1088/0266-5611/27/2/025007}
  {\path{doi:10.1088/0266-5611/27/2/025007}}.

\bibitem{Andrle2012}
M.~Andrle and A.~El~Badia.
\newblock Identification of multiple moving pollution sources in surface waters
  or atmospheric media with boundary observations.
\newblock {\em Inverse Problems}, 28(7):075009 (22pp), 2012.
\newblock \href {https://doi.org/10.1088/0266-5611/28/7/075009}
  {\path{doi:10.1088/0266-5611/28/7/075009}}.

\bibitem{Baratchart2005}
L.~Baratchart, A.~Ben~Abda, F.~Ben~Hassen, and J.~Leblond.
\newblock Recovery of pointwise sources or small inclusions in 2d domains and
  rational approximation.
\newblock {\em Inverse Problems}, 21:51 -- 74, 2005.
\newblock \href {https://doi.org/10.1088/0266-5611/21/1/005}
  {\path{doi:10.1088/0266-5611/21/1/005}}.

\bibitem{BenBelgacem:2012}
Faker Ben~Belgacem.
\newblock Identifiability for the pointwise source detection in fisher’s
  reaction-diffusion equation.
\newblock {\em Inverse Problems}, 28(6):065015 (16pp), 2012.
\newblock \href {https://doi.org/10.1088/0266-5611/28/6/065015}
  {\path{doi:10.1088/0266-5611/28/6/065015}}.

\bibitem{Bruckner2000}
Gottfried Bruckner and Masahiro Yamamoto.
\newblock Determination of point wave sources by pointwise observations:
  stability and reconstruction.
\newblock {\em Inverse Problems}, 16:723 -- 748, 2000.
\newblock \href {https://doi.org/10.1088/0266-5611/16/6/30723}
  {\path{doi:10.1088/0266-5611/16/6/30723}}.

\bibitem{Chen2022}
Zhiming Chen, Wenlong Zhang, and Jun Zou.
\newblock Stochastic convergence of regularized solutions and their finite
  element approximations to inverse source problems.
\newblock {\em SIAM J. Numer. Anal.}, 60:751 -- 780, 2022.
\newblock \href {https://doi.org/10.1137/21M1409779}
  {\path{doi:10.1137/21M1409779}}.

\bibitem{Faria2020}
J.~Rocha de~Faria.
\newblock A genetic algorithm for pointwise source reconstruction by the method
  of fundamental solutions.
\newblock {\em Trends in Computational and Applied Mathematics}, 23(3):401 --
  412, 2022.
\newblock \href {https://doi.org/10.5540/tcam.2022.023.03.00401}
  {\path{doi:10.5540/tcam.2022.023.03.00401}}.

\bibitem{EganMahoney:1972}
Bruce~A. Egan and James~R. Mahoney.
\newblock Numerical modeling of advection and diffusion of urban area source
  pollutants.
\newblock {\em Journal of Applied Meteorology and Climatology}, 11(2):312 --
  322, 1972.
\newblock URL:
  \url{https://journals.ametsoc.org/view/journals/apme/11/2/1520-0450_1972_011_0312_nmoaad_2_0_co_2.xml},
  \href {https://doi.org/10.1175/1520-0450(1972)011<0312:NMOAAD>2.0.CO;2}
  {\path{doi:10.1175/1520-0450(1972)011<0312:NMOAAD>2.0.CO;2}}.

\bibitem{ElBadia2000}
A.~El~Badia and T.~Ha-Duong.
\newblock An inverse source problem in potential analysis.
\newblock {\em Inverse Problems}, 16(3):651 -- 663, 2000.
\newblock \href {https://doi.org/10.1088/0266-5611/16/3/308}
  {\path{doi:10.1088/0266-5611/16/3/308}}.

\bibitem{GrebenkovNguyen:2013}
D.~S. Grebenkov and B.-T. Nguyen.
\newblock Geometrical structure of {L}aplacian eigenfunctions.
\newblock {\em SIAM Rev.}, 55(4):601--667, 2013.
\newblock \href {https://doi.org/10.1137/120880173}
  {\path{doi:10.1137/120880173}}.

\bibitem{Hu2024}
Tianhao Hu, Bangti Jin, and Zhi Zhou.
\newblock Point source identification using singularity enriched neural
  networks.
\newblock {\em arXiv:2408.09143v1 [math.NA]}, 2024.
\newblock URL: \url{https://arxiv.org/abs/2408.09143v1}.

\bibitem{Kandaswamy2009}
D.~Kandaswamy, T.~Blu, and D.~Van De~Ville.
\newblock Analytic sensing: Noniterative retrieval of point sources from
  boundary measurements.
\newblock {\em SIAM J. SCI. COMPUT.}, 31(4):3179 -- 3194, 2009.
\newblock URL: \url{https://epubs.siam.org/doi/10.1137/080712829}, \href
  {https://doi.org/10.1137/080712829} {\path{doi:10.1137/080712829}}.

\bibitem{Kovalets2011}
Ivan~V. Kovalets, Spyros Andronopoulos, Alexander~G. Venetsanos, and John~G.
  Bartzis.
\newblock Identification of strength and location of stationary point source of
  atmospheric pollutant in urban conditions using computational fluid dynamics
  model.
\newblock {\em Mathematics and Computers in Simulation}, 82(2):244 -- 257,
  2011.
\newblock \href {https://doi.org/10.1016/j.matcom.2011.07.002}
  {\path{doi:10.1016/j.matcom.2011.07.002}}.

\bibitem{LeNiliot2004}
Christophe Le~Niliot and Frederic Lefevre.
\newblock A parameter estimation approach to solve the inverse problem of point
  heat sources identification.
\newblock {\em International Journal of Heat and Mass Transfer}, 47(6):827 --
  841, 2004.
\newblock \href {https://doi.org/10.1016/j.ijheatmasstransfer.2003.08.011}
  {\path{doi:10.1016/j.ijheatmasstransfer.2003.08.011}}.

\bibitem{LiLiuShi:2023}
Jialei Li, Xiaodong Liu, and Qingxiang Shi.
\newblock Reconstruction of multiscale elastic sources from multifrequency
  sparse far field patterns.
\newblock {\em SIAM J. Appl. Math.}, 83(5):1915--1934, 2023.
\newblock \href {https://doi.org/10.1137/22M1544257}
  {\path{doi:10.1137/22M1544257}}.

\bibitem{LiYangZhangZhang:2021}
Long Li, Jiansheng Yang, Bo~Zhang, and Haiwen Zhang.
\newblock Imaging of buried obstacles in a two-layered medium with phaseless
  far-field data.
\newblock {\em Inverse Problems}, 37(5):Paper No. 055004, 26, 2021.
\newblock \href {https://doi.org/10.1088/1361-6420/abec1d}
  {\path{doi:10.1088/1361-6420/abec1d}}.

\bibitem{LiZhang:2020}
Zhiyuan Li and Zhidong Zhang.
\newblock Unique determination of fractional order and source term in a
  fractional diffusion equation from sparse boundary data.
\newblock {\em Inverse Problems}, 36(11):115013, 20, 2020.
\newblock \href {https://doi.org/10.1088/1361-6420/abbc5d}
  {\path{doi:10.1088/1361-6420/abbc5d}}.

\bibitem{LinZhangZhang:2022}
Guang Lin, Zecheng Zhang, and Zhidong Zhang.
\newblock Theoretical and numerical studies of inverse source problem for the
  linear parabolic equation with sparse boundary measurements.
\newblock {\em Inverse Problems}, 38(12):Paper No. 125007, 28, 2022.

\bibitem{Ling:2006}
Leevan Ling, Masahiro Yamamoto, Y.~C. Hon, and Tomoya Takeuchi.
\newblock Identification of source locations in two - dimensional heat
  equations.
\newblock {\em Inverse Problems}, 22:1289 -- 1305, 2006.
\newblock URL: \url{http://iopscience.iop.org/0266 - 5611/22/4/011}, \href
  {https://doi.org/10.1088/0266 - 5611/22/4/011} {\path{doi:10.1088/0266 -
  5611/22/4/011}}.

\bibitem{LiuMeng:2023}
Xiaodong Liu and Shixu Meng.
\newblock A multi-frequency sampling method for the inverse source problems
  with sparse measurements.
\newblock {\em CSIAM Trans. Appl. Math.}, 4(4):653--671, 2023.

\bibitem{Liu2023}
Xiaodong Liu and Qingxiang Shi.
\newblock Identification of acoustic point sources in a two-layered medium from
  multi-frequency sparse far field patterns.
\newblock {\em Inverse Problems}, 39(6):065001 (30pp), 2023.
\newblock \href {https://doi.org/10.1088/1361-6420/accaa0}
  {\path{doi:10.1088/1361-6420/accaa0}}.

\bibitem{Nakaguchi2012}
Etsushi Nakaguchi, Hirokazu Inui, and Kohzaburo Ohnaka.
\newblock An algebraic reconstruction of a moving point source for a scalar
  wave equation.
\newblock {\em Inverse Problems}, 28(6):065018 (21pp), 2012.
\newblock \href {https://doi.org/10.1088/0266-5611/28/6/065018}
  {\path{doi:10.1088/0266-5611/28/6/065018}}.

\bibitem{Ren2019}
Kui Ren and Yimin Zhong.
\newblock Imaging point sources in heterogeneous environments.
\newblock {\em arXiv:1901.07189v2 [math.AP]}, 2019.
\newblock URL: \url{https://arxiv.org/abs/1901.07189v2}, \href
  {https://doi.org/10.1088/0901.07189v2} {\path{doi:10.1088/0901.07189v2}}.

\bibitem{Faria2022}
Jairo Rocha~de Faria, Daniel Lesnic, Rômulo da~Silva~Lima, and Thiago~José
  Machado.
\newblock The method of fundamental solutions for pointwise source
  reconstruction.
\newblock {\em Computers and Mathematics with Applications}, 114:171 -- 179,
  2022.
\newblock \href {https://doi.org/10.1016/j.camwa.2022.03.041}
  {\path{doi:10.1016/j.camwa.2022.03.041}}.

\bibitem{RundellZhang:2020}
William Rundell and Zhidong Zhang.
\newblock On the identification of source term in the heat equation from sparse
  data.
\newblock {\em SIAM J. Math. Anal.}, 52(2):1526--1548, 2020.
\newblock \href {https://doi.org/10.1137/19M1279915}
  {\path{doi:10.1137/19M1279915}}.

\bibitem{Siegel:2014}
Carl~L. Siegel.
\newblock \"{U}ber einige {A}nwendungen diophantischer {A}pproximationen
  [reprint of {A}bhandlungen der {P}reu\ss ischen {A}kademie der
  {W}issenschaften. {P}hysikalisch-mathematische {K}lasse 1929, {N}r. 1].
\newblock In {\em On some applications of {D}iophantine approximations},
  volume~2 of {\em Quad./Monogr.}, pages 81--138. Ed. Norm., Pisa, 2014.

\bibitem{SunZhang:2022}
Chunlong Sun and Zhidong Zhang.
\newblock Uniqueness and numerical inversion in the time-domain fluorescence
  diffuse optical tomography.
\newblock {\em Inverse Problems}, 38(10):Paper No. 104001, 23, 2022.

\bibitem{Wang2023}
Zhongjian Wang, Wenlong Zhang, and Zhiwen Zhang.
\newblock A data-driven model reduction method for parabolic inverse source
  problems and its convergence analysis.
\newblock {\em Journal of Computational Physics}, 487:112156, 2023.
\newblock \href {https://doi.org/10.1016/j.jcp.2023.112156}
  {\path{doi:10.1016/j.jcp.2023.112156}}.

\bibitem{Zhang2023}
Deyue Zhang, Yan Chang, and Guo，Yukun.
\newblock Jointly determining the point sources and obstacle from cauchy data.
\newblock {\em arXiv:2306.06665v1 [math.NA]}, Not applicable, 2023.
\newblock URL: \url{https://arxiv.org/abs/2306.06665v1}, \href
  {https://doi.org/10.1088/2306.06665v1} {\path{doi:10.1088/2306.06665v1}}.

\end{thebibliography}

\end{document}